\documentclass[11pt,a4paper,reqno]{amsart}%
\usepackage{amsthm,amsmath,amsfonts,amssymb,amsxtra,appendix,bookmark,dsfont,latexsym}

\makeatletter
\usepackage{hyperref}
\usepackage{delarray,a4,color}
\usepackage{euscript}
\usepackage[latin1]{inputenc}

\usepackage{stmaryrd}
\usepackage{amsthm,amsmath, amsfonts,amssymb,amsxtra,bookmark,latexsym,hyperref,delarray,a4,color,euscript,mathrsfs,textcomp,bm}
\usepackage[latin1]{inputenc}
\usepackage{enumerate}
\usepackage{amsmath,amssymb}
\usepackage{tikz}
\usepackage{comment}
\newcommand{\vertiii}[1]{{\left\vert\kern-0.25ex\left\vert\kern-0.25ex\left\vert #1 
    \right\vert\kern-0.25ex\right\vert\kern-0.25ex\right\vert}}

\newcommand{\llrrbrace}[1]{
  \left\{\mkern-6mu\left\{#1\right\}\mkern-6mu\right\}}

\newtheorem{lemma}{Lemma}[section]
\newtheorem{remark}{Remark}[section]
\newtheorem{theorem}{Theorem}[section]

\newtheorem{proposition}{Proposition}[section]
\theoremstyle{remark}

\let\oldremark\remark
\renewcommand{\remark}{\oldremark\normalfont}

\usepackage{amsmath, amssymb, bm}
\usepackage{latexsym}
\usepackage[latin1]{inputenc}
\usepackage[T1]{fontenc}
\usepackage{multicol}
\usepackage{listings}
\usepackage{stmaryrd}
\usepackage{amsthm,amsmath, amsfonts,amssymb,amsxtra,bookmark,latexsym,hyperref,delarray,a4,color,euscript,mathrsfs,textcomp,bm,comment}
\usepackage[latin1]{inputenc}
\usepackage{enumerate}
\usepackage{amsmath,amssymb}
\usepackage{tikz}
\usepackage{latexsym}
\usepackage[latin1]{inputenc}
\usepackage[T1]{fontenc}
\usepackage{multicol}
\usepackage{listings}
\numberwithin{equation}{section}
    
\theoremstyle{remark}

\theoremstyle{definition}
\newtheorem*{definition}{Definition}

\begin{document}
\title[A posteriori error estimates for DG method to elasticity problem]{A posteriori error estimates for discontinuous Galerkin method to the elasticity problem
}
%
\author{LUONG Thi Hong Cam}\address{Department of Mathematics, University of Cergy-Pontoise, France}\email{thi-hong-cam.luong@u-cergy.fr}
\author{DAVEAU Christian}\address{Department of Mathematics, University of Cergy-Pontoise, France}\email{christian.daveau@u-cergy.fr}
\date{\today}
\begin{abstract} 
This work  concerns with  the discontinuous Galerkin (DG) method for the time-dependent linear elasticity problem. We  derive the {\em a posteriori} error bounds for semi-discrete and fully discrete problems, by making use of the {\em stationary elasticity reconstruction} technique which allows to estimate the error  for time-dependent problem through the error estimation of the associated stationary elasticity problem. To this end, to derive the error bound for the stationary problem, we present two  methods to obtain two different {\em a posteriori} error bounds, by $L^2$ duality technique and {\em via} energy norm.  For fully discrete scheme, we make use of the backward-Euler scheme and an appropriate space-time reconstruction. The technique  here can be applicable for a variety of DG methods as well.
\end{abstract}

\maketitle
\section*{Introduction}

{\em A posteriori} error estimation and adaptivity recently  have become successful tools for efficient numerical computations.  Many adaptive methods have  already been widely considered for solving elliptic, parabolic and first order hyperbolic problems. For example, in~\cite{H-P-D} they consider  the {\em a posteriori} error estimates for Maxwell equation,  \cite{K-M} considers {\em a posteriori} for time dependent Stokes equation,\cite{G-L-V} investigates the {\em a posteriori} error estimate for parabolic problem.
Howerver, there are few results about {\em a posteriori} error analysis for the second order evolution problems; we mention, in particular,~\cite{B-S} and ~\cite{G-L-M} derive rigorous  {\em a posteriori} bounds for conforming finite element methods in case of fully discretization for wave equation. 

 An important class of  finite element methods goes
under the general term discontinuous Galerkin  (DG) methods which approximate the solutions by means of piecewise continuous functions.
In this work, we study the DG methods for elasticity problem with the derivation of {\em a posteriori} error bounds in the $L^\infty (L^2 )$ norm.
Concerning with the study of {\em a posteriori} error estimates for elasticity system, several different approaches have been developed for stationary problem: we refer to ~\cite{Ver} and \cite{C-H-H-X} for residual type error estimators and method based on dual variables, respectively; on the other hand, in~\cite{C-H-H-X}  and~\cite{Wihler} they consider  {\em a posteriori} error estimates  for isotropic problem.  In case of time-dependent problem, to our knowledge there is still no work for {\em a posteriori} error control for elasticity equation, so this is the first study about this. Inspired from the recent work of ~\cite{G-L-M} on {\em a posteriori} error estimates for finite element method for wave equation, we here propose an {\em a posteriori} bound in the $L^\infty(L^2)$ norm of the error  for the elasticity equations, and in case of DG method. The theory is developed for both space-discrete case, as well as for the  case of  fully discrete scheme with the implicit backward Euler method in time discretization. The main techniques for this derivation conclude the special testing procedure introduced by ~\cite{Baker}, and adapt the technique about elliptic reconstruction used in~\cite{M-N}, ~\cite{La-Mak} to stationary elasticity ({SE}) reconstruction. The {SE} reconstruction technique allows to estimate the error of the time-dependent problem {\em via} {\em an auxiliary SE equation}.

\section{Notations and approximation properties}
\subsection{Some functional spaces}
Throughout this paper, the usual Sobolev norm of $H^m$ on $E\subset {\mathbb R}^d$ is denoted by  $\|\cdot\|_{0,E}=\|v\|_{L^2(E)}$ , the standard space, norm, and inner product notation
are adopted, their definitions can be found for example in~\cite{A-F}.

For  the spaces of vector ${\bm v}=(v_i)_{1\leq i\leq d}, {\bm w}=(w_i)_{1\leq i\leq d} \in X(E)^d$ and  tensor fields  ${\bm \sigma}=(\sigma_{ij})_{1\leq i,j\leq d}, {\bm \tau}=(\tau_{ij})_{1\leq i,j\leq d }\in X(E)^{d\times d}$, we define the following scalar products:
$${\bm v } \cdot {\bm w}=\sum_{i=1}^d v_iw_i,~~~~~~~~{\bm \sigma}: {\bm \tau}=\sum_{i,j=1}^d \sigma_{ij}\tau_{ij};$$
and the norms 
$$|\bm v|=\sqrt{{\bm v } \cdot {\bm v}},~~~~|\bm \sigma|=\sqrt{{\bm \sigma } : {\bm \sigma}}.$$
In case  $\mbox{\boldmath$v$}$ is a vector-valued or matrix-valued function, the corresponding term $\|\mbox{\boldmath$v$}\|_{ s,E}$ is defined
in a similar manner. Here we use the same notations for the norms of vector-valued and matrix-valued.\\

Let $X$ is a Banach space, we will make use of the Bochner space $L^p (0,T; X)$, $1\leq p \leq \infty$
endowed with the norm
$$
\|\mbox{\boldmath$u$}\|_{L^p(0,T,X)}=\left \{
 \begin{array}{ll}
 (\int_0^T \|\mbox{\boldmath$u$}\|_{X}^p )^{1/p},~~&1 \leq p< \infty;\\
  {\rm ess~ sup}_{t\in [0,T]} \|\mbox{\boldmath$u$}\|_{X},~~&p=\infty.
 \end{array}
\right.
$$
\subsection{Triangulation}
Let ${\mathcal{T}}_h$ be a subdivision of  $\Omega$ into disjoint open
sets $\{K\}$ such that $\overline{\Omega}=\cup_{K\in {\mathcal{T}}_h} \overline{K}$, and assume that the mesh is {\em regular}, e.g. it has no hanging
nodes. Denote by $h_K:={\rm {diam}} (K)$ the diameter of element $K$ and by $h:=\max_{K\in {\mathcal T}_h} h_K$ the mesh size.

We assume that the triangulation to be {\em shape-regular}, i.e.  there exists a constant $\vartheta >0$ such that
$\vartheta:=\sup_{K\in {\mathcal T}_h}  \frac{h_K}{\rho_K} <\infty,$
here $h_K$ and $\rho_K$ denote the diameter of $K$ and the diameter of the largest ball inscribed in $K$, respectively.

The above assumptions  also imply that the mesh is {\em locally quasi-uniform}, that is, there is  a constant $\kappa >0$  such that
$\kappa \leq \frac{h_K}{h_{K'}} \leq \kappa^{-1},$
whenever $K$ and $K'$ share a common edge.
\subsection{Finite element spaces}
Given a partition or mesh ${\mathcal T}_h$ of $\Omega$, we define the following so-called broken Sobolev space for $s\geq 0$:
\begin{equation}
H^s({\mathcal T}_h)^d=\{{\bm v} \in L^2(\Omega)^d:{\bm v}|_K\in H^s(K)^d, ~K\in {\mathcal T}_h\},
\end{equation}
and endow it with the norm
$$\|{\bf v}\|_{s,{\mathcal T}_h}=\big (\sum_{K\in {\mathcal T}_h}\|{\bf v}\|_{s,K}^2\big )^{1/2}. $$
 
 Consider the DG finite element space
\begin{equation}\label{spacevh}
{\bf V}_h := \{\mbox{\boldmath$v$}\in L^2 (\Omega)^d : \mbox{\boldmath$v$}|_K  \in {\mathbb P}_r ({K})^d,~ K \in {\mathcal T}_h\},
\end{equation}
here   ${\mathbb P}_r ({K})$ denotes  the set  of polynomials of total degree at most $r$ on $K$:
\subsection{Averages and jumps}
Let ${\mathcal E}_h$ be a union of all faces (edges) of the triangulation ${\mathcal T}_h$, and denote by ${\mathcal{E}}_h={\mathcal{E}}_h^I\cup {\mathcal{E}}_h^B$ with ${\mathcal E}^I_h$ the union of all interior faces (edges) of the  triangulation ${\mathcal T}_h$,
 and  $\mathcal{E}_h^B$ being the set of all boundary faces.
Here we will generically refer to any element of $\mathcal{E}_h$ as a face, in both two and three
dimensions.\\

Let $K^+$ and $K^-$  be two adjacent elements of ${\mathcal T}_h$. Let $\bm x$ be an arbitrary point on the common side  $e=\partial K^+ \cap \partial K^-$. For a vector-valued function ${\bm q} \in H^s({\mathcal T}_h)^d,~s>\frac{1}{2}$, let us denote by $\bm q ^{\pm}$ the trace of $\bm q$ on $e$ from the interior of $K^{\pm}$. Then we define average and jump at ${\bm x} \in e$ as follows:

$$\llrrbrace{\mbox{\boldmath$q$}}=\frac{1}{2}(\mbox{\boldmath$q$}^ ++\mbox{\boldmath$q$}^-);\quad \llbracket{\mbox{\boldmath$q$}}\rrbracket=(\mbox{\boldmath$q$}^ +-\mbox{\boldmath$q$}^-).$$
If $e$ is a boundary side ($e \in  {\mathcal{E}}_h^B$), these definitions are modified to
$$\llrrbrace{\mbox{\boldmath$q$}}=\mbox{\boldmath$q$};~~~~\llbracket \mbox{\boldmath$q$}\rrbracket=\mbox{\boldmath$q$}.$$

Introducing the function $\sf h$ defined on $e\in {\mathcal E}_h$  by

 $${\sf h}|_e=\left \{
 \begin{array}{ll}
  \mbox{min} \{h_K,h_{K'}\},& \quad e \in {\mathcal{E}}^I_h,~~~~ e=\partial K\cap \partial K', \\
  h_K,&\quad e \in {\mathcal{E}}^B_h,~~~~ e=\partial K\cap \partial \Omega.
 \end{array}
 \right.
$$
\subsection{Auxiliary inequalities}
\begin{lemma}\label{invconstant}{\rm (Inverse inequality, cf.~\cite{Riv})}. Let $K$ be an element of ${\mathcal T}_h$. Then, there exists constant $C_{\rm{inv}}$  independent of $h_K$, $v$ but depend on the polynomial degree $r$ such that
\begin{enumerate}[i)]
\item
$\forall v\in {\mathbb P}_r(K), \forall e \in \partial K, \|v\|_{0,e}\leq C_{\rm{inv}}h_K^{-1/2}\|v\|_{0,K},$
\item $\forall v\in {\mathbb P}_r(K), \forall e \in \partial K, \|\nabla v\cdot {\bm \nu}_e\|_{0,e}\leq C_{\rm{inv}} h_K^{-1/2}\|\nabla v\|_{0,K}.$
\end{enumerate}
\end{lemma}
By $|\cdot|_{j,K}$ we denote the semi-norm of $H^j(K)$. Then we also have:
\begin{lemma}\label{inv}{\rm (cf.~\cite{Riv})}.
There exists a constant $C$ depending only on the shape-regularity of the mesh, the approximation order $r$, and the dimension $d$
$$|v|_{j,K} \leq C h_k^{i-j}|v|_{i,K},\quad \forall v\in {\mathbb P}_r(K),~0\leq i\leq j\leq 2.$$
\end{lemma}
\begin{lemma}\label{bestapp}({\rm cf. \cite{Riv}, Theorem 2.6}).
Let $K$ be an element of the triangulation. Let $v\in H^s(K)$ for $s\geq 1$ and let $r\geq 0$ be an integer. There exists a constant $C$ independent of $v$ and $h_K$ and a function ${\tilde v}\in {\mathbb P}_r(K)$ such that
$$\forall 0\leq q\leq s, \|v-{\tilde v}\|_{q,K}\leq Ch_K^{\mu-q}|v|_{s,K};$$
where $\mu={\min}(r+1,s)$ and $e \in \partial K$.
\end{lemma}
We also have the following approximation properties (cf.~\cite{R-W-G}):
\begin{lemma}\label{bestapp2}
Let $K$ be an element of the triangulation with $h_K$ its diameter. Let $v\in H^s(K)$, let $r\geq 0$ be an integer. There exists a constant $C$ independent of $v$ and $h_K$ and a function ${\tilde v}\in {\mathbb P}_r(K)$ such that
\begin{subequations}\label{appv3}
\begin{alignat}{3}
\forall 0\leq q\leq s,& \|v-{\tilde v}\|_{q,K}\leq Ch_K^{\mu-q}\|v\|_{s,K},\quad s\geq 0,\\
\forall 0\leq q\leq s,& \|v-{\tilde v}\|_{0,e}\leq Ch_K^{\mu -1/2}\|v\|_{s,K},\quad s>1,\\
\forall 0\leq q\leq s, &\|v-{\tilde v}\|_{1,e}\leq Ch_K^{\mu -3/2}\|v\|_{s,K},\quad s>3/2,
\end{alignat}
\end{subequations}
where $\mu={\min}(r+1,s)$ and $e \in \partial K$.
\end{lemma}
We also have the trace inequality:
\begin{lemma}\label{tr}{\rm (Trace inequality, cf.~\cite{K-Pb})}.
Let $K$ be an element of the triangulation, and denote by $h_K$ the diameter of $K$. Let $e \in \partial K$, the following inequality holds
$$\|v\|^2_{0,\partial K}\leq C_{\rm tr}(h_K^{-1}\|v\|_{0,K}^2+h_K\|\nabla v\|_{0,K}^2), \quad \forall v \in H^1(K),$$
here $C_{\rm tr}$ is independent of $h_K$ and $v$. 
\end{lemma}
\begin{remark}
Note that analogous bounds as in the lemmas above can be also obtained for vector-valued function.
\end{remark}
\section{Mathematical model and DG formulation}
\subsection{Mathematical model}
We consider the stress tensor  ${\bm \sigma}({\bm u})=(\sigma_{ij} (\mbox{\boldmath$u$}))_{1\leq i,j \leq d}$,
which satisfies the constitutive relationship:
${\bm \sigma}({\bm u})= {\mathbb C}{\bm \varepsilon}({\bm u}),$ or equivalently 
$$ \forall 1 \leq i, j \leq d,
\sigma_{ij} (\mbox{\boldmath$u$}) =\sum_{k,l=1}^d
C_{ij kl}
\varepsilon_{kl}(\mbox{\boldmath$u$}),$$
where ${\mathbb C} = (C_{ij kl} )_{ij kl}$ is a fourth order tensor, independent of $t$ and satisfying some symmetry properties:
\begin{equation}\label{tensor1}
C_{ij kl} = C_{j ikl} = C_{ij lk} = C_{klij}.
\end{equation}
We assume that the stiffness tensor $\mathbb C$ is positive definite and piecewise constant in $\Omega$, then there exists a constant $c_*>0$ such that
\begin{equation}\label{tensor}\forall (\gamma_{ij} )_{ij} \neq 0,~~
 \sum_{1\leq i,j,k,l\leq d} C_{ij kl} \gamma_{kl} \gamma_{ij} \geq c_{*} \sum_{1\leq i,j\leq d} \gamma_{ij}^2  >0.
\end{equation}
\begin{remark}\label{C^*}
With piecewise constant tensor $\mathbb C$, there exists a positive constant $C^*$ such that $\|{\mathbb C}\mbox{\boldmath$\tau$}\|_{0,K} \leq C^*\|\mbox{\boldmath$\tau$}\|_{0,K} $  for any element $K\in {\mathcal T}_h$, and matrix $\mbox{\boldmath$\tau$} \in {L^2(\Omega)^{d\times d}}$.
\end{remark}
 We consider the equations of linear elasticity problem of finding the displacement vector ${{\bm u}}=(u_i({\bm x},t))_{i=1}^d $  at a point ${\bm x}$ in $\Omega$ and at a time $t\in [0,T]$  such that
\begin{equation}\label{eq1}
\left \{
\begin{array}{ll}
 \rho\partial^2 _{tt}{ u_i}-\sum_{j=1}^d \dfrac{\partial \sigma_{ij}(\mbox{\boldmath$u$})}{{\partial }{ x_j}} = f_i & \mbox{in}~  (0,T) \times \Omega,~~~~i=\overline{1,d},\\
 \mbox{\boldmath$u$}= {\bf 0} & \mbox{on}~ \partial \Omega \times (0,T],
\end{array} \right.
\end{equation}
 with initial conditions $\mbox{\boldmath$u$}_0 \in H^1_0(\Omega)^d$, and $\mbox{\boldmath$u$}_1 \in L^2(\Omega)^d$:
\begin{equation}\label{initialcon}
\left \{
\begin{array}{ll}
\mbox{\boldmath$u$} (\cdot, 0)= \mbox{\boldmath$u$}_0  &\mbox{on}~ \Omega \times \{0\},\\
\partial_t\mbox{\boldmath$u$} (\cdot, 0)= \mbox{\boldmath$u$}_1  &\mbox{on}~ \Omega \times \{0\}.
\end{array} \right.
\end{equation}
Here   ${\bm f}({\bm x},t)=(f_i({\bm x},t))_{i=1}^d \in L^2(0,T;L^2(\Omega)^d)$ is a general source function; and $\rho ({\bm x})$ is the mass density of the material.
For simplicity, we will take the coefficient $\rho=1$. 

Assume that the analytical solution ${\bm u}$ in~(\ref{eq1}) satisfies  the following variational formulation:
\begin{equation}\label{weak}
(\partial ^2_{tt}{{\bm u}},{{{\bm v}}} )_\Omega + a({{\bm u}},{{\bm v}}) = l({{\bm v}}),~ \forall \mbox{\boldmath$v$} \in H^1_0(\Omega)^d,~\mbox{ a.e. in}~ (0,T),
\end{equation}
where the bilinear form $a(\cdot,\cdot)$ is defined by
\begin{equation}
a({{\bm u}},{{\bm v}})=(\mbox{\boldmath$\sigma$}(\mbox{\boldmath$u$}), {\bm \varepsilon}({{\bm v}}))_\Omega=\int_\Omega \sum_{i,j=1}^d \sigma_{ij}(\mbox{\boldmath$u$}) \varepsilon_{ij}({{\bm v}}){\rm d}{\bm x},
\end{equation}
and the linear form $l(\cdot)$ is given by
\begin{equation}
l({{\bm v}})=\int_{\Omega}{\bm f} \cdot {{\bm v}}{\rm d}{\bm x}.
\end{equation}

\subsection{DG formulation}\begin{definition}\label{auhvinitial}
Define the bilinear form $a_h$ on $H^s({\mathcal T}_h)^d\times H^s({\mathcal T}_h)^d$, $s> 3/2$ by
\begin{equation}\label{ahuv}
\begin{split}
a_h(\mbox{\boldmath$u$},\mbox{\boldmath$v$} )&=\sum_{K\in {\mathcal{T}}_h} \int_{K}\mbox{\boldmath$\sigma$}({\bm u}):{\bm \varepsilon}({\bm v}){\rm d}{\bm x}-\sum_{e\in \mathcal{E}_h} \int_{e}\llrrbrace{(\mbox{\boldmath$\sigma$}({\bm u})) \mbox{\boldmath$\nu$}_e}\cdot\llbracket{\mbox{\boldmath$v$}}\rrbracket{\rm d}A\\
 &-\sum_{e\in \mathcal{E}_h} \int_{e}\llrrbrace{(\mbox{\boldmath$\sigma$}({\bm v})){\bm \nu}_e}\cdot\llbracket{\mbox{\boldmath$u$}}\rrbracket{\rm d}A+\sum_{e\in \mathcal{E}_h} \int_{e}{\texttt a}\llbracket{\mbox{\boldmath$u$}}\rrbracket\cdot\llbracket{\mbox{\boldmath$v$}}\rrbracket{\rm d}A;
 \end{split}
 \end{equation}
where $\mbox{\boldmath$\nu$}_e$ is an unit normal vector associated to the face $e$, oriented from $K^+$ to $K^-$.
\end{definition}
 In formulation~(\ref{ahuv}) above, the function {\texttt a} in the last terms penalies the jumps of ${\bm u}$ and ${\bm v}$ over the faces of ${\mathcal T}_h$; ${\texttt a}=\alpha {\sf h}^{-1}$ with $\alpha$ is a positive parameter independent of the local mesh sizes will be specified later to assure the coercivity of the bilinear form;
$\mbox{\boldmath$\sigma$}$ and $\mbox{\boldmath$\varepsilon$}$ are the stress tensor and symmetric gradient respectively, taken element-wise.

 The  semi-discrete DG approximation to~(\ref{weak}) then reads as follows: 
  Find ${{\bm u}_h}: [0,T] \rightarrow {\bf V}_h$ such that 
 \begin{equation}\label{DG1}
 ( \partial^2_{tt}{{\bm u}_h}, {{\bm v}})_\Omega + a_h({\bm u}_h,{\bm v} ) =( {\mbox{\boldmath$f$}},{{\bm v}} )_\Omega ~\mbox{for~all}~ {{\bm v}} \in {\bf V}_h, t \in (0,T],
 \end{equation}
with initial values
\begin{equation}\label{initialconditionuh}
\begin{split}
{{\bm u}_h}|_{t=0} &= {\sf \Pi}_h {{\bm u}}_0,\\
\partial_t{{\bm u}_h} |_{t=0}& = {\sf \Pi}_h {{\bm u}}_1,
\end{split}
\end{equation}
here ${\sf \Pi}_h$ is the orthogonal $L^2$-projection  onto ${\bf V}_h$, and the discrete bilinear form $a_h$ on
${\bf V}_h \times {\bf V}_h$ is given in~(\ref{ahuv}).

\begin{lemma}\label{coerah}{\rm (Well-posedness of the bilinear form $a_h$)}.
 Let the interior penalty parameter  be defined as in~(\ref{ahuv}). Then there are  positive constants $\kappa=\min\{\frac{c_*}{2},\frac{1}{2}\}$, $M=\max\{2, 2\alpha^{-1/2} C_{\rm{inv}}C^*+ 2C^*\} $, and $\alpha_{\rm {min}}=4C_{\rm inv}^2(C^*)^2c_*^{-1}$ independent of $h$ such that 
$$|a_h ({{\bm u}},{{\bm v}})| \leq M \vertiii{{\bm u}} \cdot \vertiii{{\bm v}},~~~\forall {{\bm u}},{{\bm v}} \in {\bf V}_h,$$
and for $\alpha \geq \alpha_{\rm {min}}$, we have that
$$a_h ({{\bm u}},{{\bm u}}) \geq \kappa \vertiii{{\bm u}}^2,~~~\forall {{\bm u}} \in {\bf V} _h,$$
with DG norm
 $$\vertiii{{\bm u}}=\big(\sum_{K \in {\mathcal{T}}_h}\|{\bm \varepsilon({{\bm u}})}\|_{0,K}^2+\sum_{e\in {\mathcal E}_h} \alpha {\sf h}^{-1} \|\llbracket{{{\bm u}}}\rrbracket\|^2_{0,e}\big)^{1/2}.$$
\end{lemma}
\section{Semi-discrete error estimates}
\subsection{Stationary elasticity reconstruction}\label{statrecosemi}
In this section, we introduce the stationary reconstruction which is the main tool to derive the error of the time-dependent problem.
\begin{definition}\label{reconstw} { ({SE} reconstruction and error splitting)}. Let ${\bm u}_h$ be the  DG approximation given by~(\ref{DG1}). Let also ${\sf \Pi}_h : L^2 (\Omega)^d \rightarrow {\bf V}_h$ be the orthogonal $L^2$ -projection operator onto the
finite element space ${\bf V}_h$. We define the {SE} reconstruction ${\bm w}= {\bm w}(t) \in H^1_{0}(\Omega)^d$ of ${\bm  u}_h ={\bm   u}_h (t)$ at time $t \in [0, T ]$ as
the solution of the {SE} problem
\begin{equation}\label{recon-semi}
a({\bm w},\mbox{\boldmath$v$} ) = ( \mbox{\boldmath$g$}, \mbox{\boldmath$v$})_{\Omega} ~ \mbox{for all}~ \mbox{\boldmath$v$}\in H^1_{0}(\Omega)^d,
\end{equation}
where
\begin{equation}\label{gcon}
{\bm g} := {\bm   B} {\bm  u}_h - {\sf \Pi}_h \mbox{\boldmath$f$} + \mbox{\boldmath$f$},
\end{equation}
and ${\bm  B} : {\bf V}_h \rightarrow {\bf V}_h$ is the discrete  operator defined by
$$( {\bm  B}{\bm z}, {\bm v} )_\Omega = a_h({\bm z}, {\bm v} )~\mbox{for all} ~~~ {\bm v} \in {\bf V}_h,$$
for each ${\bf z} \in {\bf V}_h$.\\

We decompose the error as
\begin{equation}\label{errordecompose}
\mbox{\boldmath$e$} := {\bm  u}_h- \mbox{\boldmath$u$} = \mbox{\boldmath$\rho$}-{\bm \theta},
\end{equation}
where $\mbox{\boldmath$\theta$}:= {\bm w} -{\bm  u}_h$ and $\mbox{\boldmath$\rho$} := {\bm w}- \mbox{\boldmath$u$} \in H^1_0(\Omega)^d$.
\end{definition}
\begin{remark}({\rm Well-posedness of $\bm w$}).\label{rolew}
The {SE} reconstruction above
is well defined. Indeed, ${\bm   B} {\bm  u}_h \in {\bf V}_h$ is the unique $L^2-$Riesz representation of a linear
functional on the finite-dimensional space ${\bf V}_h$. And the existence and uniqueness of
the solution $\bm w$ in~(\ref{recon-semi}), with data ${\bm  B} {\bm  u}_h - {\sf \Pi}_h \mbox{\boldmath$f$} + \mbox{\boldmath$f$} \in L^2 (\Omega)^d$, follows from the Lax-
Milgram theorem.
\end{remark}
\begin{remark}\label{rolew22}({\rm The role of $\bm w$}).
Consider the {SE} problem of finding ${\bm w} \in {H^1_0}(\Omega)^d$ satisfying
\begin{equation}\label{eqg}
-\nabla \cdot ({\bm \sigma}({\bm w}))={\bm g},
\end{equation}
with ${\bm g}$ defined by~(\ref{gcon}).
Let ${\bm w}_h \in {\bf V}_h$ be the DG approximation to $\bm w$, defined by the
finite-dimensional linear system
$$a_h({\bm  w}_h, {\bm  v} ) = ( {\bm  B} {\bm  u}_h - {\sf \Pi}_h {\bm f} + {\bm f},{\bm  v} )_{\Omega},~\mbox{ for all}~~ {\bm  v} \in {\bf V}_h,$$
 this implies $a_h({\bm w}_h, {\bm  v} ) = ( {\bm B} {\bm u}_h, {\bm  v})_\Omega = a_h({\bm  u}_h, {\bm  v} )$ for all ${\bm  v} \in {\bf V}_h$, i.e.,
${\bm w}_h = {\bm  u}_h$. Therefore, the {SE} reconstruction $\bm w$ is the exact solution to the {SE}  problem~(\ref{eqg}) whose DG approximate solution is ${\bm u}_h$. 

\end{remark}

\subsection{Semi-discrete error relation}\label{errorrelationsemie}
\begin{lemma}\label{errorrelation}
  With reference to the notation of decomposition as in~(\ref{errordecompose}), the following error relation holds
  $$(\partial^2_{tt}\mbox{\boldmath$e$} , \mbox{\boldmath$v$})_\Omega + a(\mbox{\boldmath$\rho$},\mbox{\boldmath$v$}) = 0 ~\mbox{for all}~\mbox{\boldmath$v$} \in { H}^1_{0} (\Omega)^d.$$
\end{lemma}
\begin{proof}
We have the following expressions
\begin{equation}
\begin{split}
( {\partial^2_{tt}{\bm e}} , \mbox{\boldmath$v$})_\Omega + a(\mbox{\boldmath$\rho$},\mbox{\boldmath$v$}) &= ( {\partial^2_{tt}{\bm u}_h} , \mbox{\boldmath$v$})_\Omega + a({\bm w}, \mbox{\boldmath$v$}) - ( \partial^2_{tt}\mbox{\boldmath$u$} , \mbox{\boldmath$v$})_\Omega - a(\mbox{\boldmath$u$}, \mbox{\boldmath$v$})  \\
           &= (  {\partial^2_{tt}{\bm u}_h}  , \mbox{\boldmath$v$})_\Omega + a({\bm w}, \mbox{\boldmath$v$}) -( \mbox{\boldmath$f$} , \mbox{\boldmath$v$})_\Omega \\
           &=({\partial^2_{tt}{\bm u}_h} , {\sf \Pi}_h \mbox{\boldmath$v$})_\Omega + a({\bm w}, \mbox{\boldmath$v$}) - ( \mbox{\boldmath$f$} , \mbox{\boldmath$v$})_\Omega \\
           &=-a_h({\bm  u}_h, {\sf \Pi}_h \mbox{\boldmath$v$})+ (  \mbox{\boldmath$f$}, {\sf \Pi}_h \mbox{\boldmath$v$})_\Omega + a({\bm w}, \mbox{\boldmath$v$})   - ( \mbox{\boldmath$f$} , \mbox{\boldmath$v$})_\Omega \\
           &=-a_h({\bm  u}_h, {\sf \Pi}_h \mbox{\boldmath$v$}) + a({\bm w}, \mbox{\boldmath$v$}) + ( {\sf \Pi}_h \mbox{\boldmath$f$} - \mbox{\boldmath$f$} , \mbox{\boldmath$v$})_\Omega \\
           &= 0,   
\end{split}
\end{equation}
 where in the first equality, we used the decomposition~(\ref{errordecompose}), in the second equality we made use of~(\ref{weak}), in the third and fifth equalities,  the properties of the orthogonal $L^2$-projection and the formulation ~(\ref{DG1}) are used, finally the last equality follows  the identity $a_h({\bm  u}_h, {\sf \Pi}_h \mbox{\boldmath$v$}) - ( {\sf \Pi}_h \mbox{\boldmath$f$} - \mbox{\boldmath$f$} , \mbox{\boldmath$v$} )_\Omega   = a({\bm w}, \mbox{\boldmath$v$})$, which is deduced from the construction of ${\bm w}$ as in Definition~\ref{reconstw}.
\end{proof}
\subsection{Abstract semi-discrete error bound }\label{sm}
We have the following result for controlling the error $\|\mbox{\boldmath$e$}\|_{{ L}^\infty (0,T;{ L}^2 (\Omega)^d)}$ in terms of the nonconforming error $\bm \theta$:
\begin{theorem}\label{absm}
 {\rm (Abstract semi-discrete error bound)}. Let $\bm u$ and ${\bm u}_h$ be the  weak solution in~(\ref{weak}) and its DG approximation defined in~(\ref{DG1}) respectively, let $\bm  w$ be the {SE} reconstruction of ${\bm u}_h$ as in Definition~\ref{reconstw}. Applying the error decomposition~(\ref{errordecompose}) that ${\bm \rho}={\bm w}-{\bm u}$, $\mbox{\boldmath$\theta$}:= {\bm w} -{\bm  u}_h$, the following
error bound holds:
\begin{equation}\label{estimaterho}
\begin{split}
\|\mbox{\boldmath$\rho$}\|_{{ L}^\infty (0,T;{ L}^2 (\Omega)^d)}& \leq
\sqrt 2 \big(\|\mbox{\boldmath$u$}_0 - {\bm  u}_h(0)\|_{0,\Omega} +\|\mbox{\boldmath$\theta$}(0)\|_{0,\Omega}\big)\\
& + 2 \int_0^T \|\partial_t\mbox{\boldmath$\theta$}\|_{0,\Omega}{\rm d}t +2{\rm C_{F\Omega}c_*^{-1/2}}\|\mbox{\boldmath$u$}_1-\partial_t{\bm  u}_h(0)\|_{0,\Omega},
\end{split}
\end{equation}
Moreover, we also have
\begin{multline}\label{estimatee}
\|\mbox{\boldmath$e$}\|_{{ L}^\infty (0,T;{ L}^2 (\Omega)^d)} \leq \|\mbox{\boldmath$\theta$}\|_{{ L}^\infty (0,T;{ L}^2 (\Omega)^d)}+ 
\sqrt 2 \big(\|\mbox{\boldmath$u$}_0 - {\bm  u}_h(0)\|_{0,\Omega} +\|\mbox{\boldmath$\theta$}(0)\|_{0,\Omega}\big)\\
 + 2 \int_0^T \|\partial_t\mbox{\boldmath$\theta$}\|_{0,\Omega}{\rm d}t +2{C_{\rm F\Omega}c_*^{-1/2}}\|\mbox{\boldmath$u$}_1-\partial_t{\bm  u}_h(0)\|_{0,\Omega},
\end{multline}
where $C_{\rm F\Omega}$ is the constant of the Poincar\'e's inequality 
$\|\mbox{\boldmath$u$}\|_{0,\Omega}\leq  C_{\rm F\Omega} \|\nabla \mbox{\boldmath$u$}\|_{0,\Omega}$, $\forall \mbox{\boldmath$u$}\in H^ 1_0 (\Omega )^d$.
\end{theorem}
\begin{proof}
The proof of this estimate is based on the integration by parts and a test function which is introduced by ~\cite{Baker}, has been  widely used for example in \cite{G-S-S}, \cite{G-L-M}.
{\em Firstly, we will prove the bound~(\ref{estimaterho})}.
 Let  the test function $\hat{\mbox{\boldmath$v$}}: [0, T] \times \Omega \rightarrow \mathbb{R}^d$ defined by
\begin{equation}\label{testfunction}
\hat{\mbox{\boldmath$v$}}(t, \cdot) =
\int_t^{{t^*}} \mbox{\boldmath$\rho$}(s, \cdot) ds,~
t \in [0, T],
\end{equation}
for some fixed ${t^*} \in [0, T]$.  Clearly $\hat{\mbox{\boldmath$v$}} \in { H}^1_{0}(\Omega)^d$ and $\mbox{\boldmath$\rho$} \in  { H}^1_{0}(\Omega)^d$. Also, we observe that
$$\hat{\mbox{\boldmath$v$}}({t^*} , \cdot) = 0,~~ {\mbox{and}}~~ \partial_t \hat{\mbox{\boldmath$v$}} (t, \cdot) = -\mbox{\boldmath$\rho$}(t, \cdot)~~~\mbox{ a.e. in}~[0,T].$$
Set $\mbox{\boldmath$v$} = \hat{\mbox{\boldmath$v$}}$ in the error relation~(\ref{errorrelation}), integrate between $0$ and ${t^*}$ with respect to the variable $t$, and integrate by parts the first
term on the left-hand side to obtain
$$-\int_0^{{t^*}}
( \partial_t \mbox{\boldmath$e$},\partial_t \hat{\mbox{\boldmath$v$}})_\Omega {\rm d}t + (\partial_t \mbox{\boldmath$e$} ({t^*}), \hat{\mbox{\boldmath$v$}}({t^*} ))_\Omega - (\partial_t  \mbox{\boldmath$e$}(0), \hat{\mbox{\boldmath$v$}}(0))_\Omega +\int_0^{{t^*}} a(\mbox{\boldmath$\rho$}, \hat{\mbox{\boldmath$v$}}){\rm d}t = 0. $$
Due to the symmetry of the tensor $\mathbb C$, and using the fact that $\partial_t\hat{\mbox{\boldmath$v$}}=-{\bm \rho}$ and $\hat{\mbox{\boldmath$v$}}({t^*} , \cdot)=0$, integrate on $[0,T]$, we rewrite the identity as
$$\int_0^{{t^*}} \frac{1}{2}\frac{d}{dt}\|\mbox{\boldmath$\rho$}(t)\|_{0,\Omega}^2dt-\frac{1}{2}\int_0^{{t^*}}\frac{d}{dt}a(\hat{\mbox{\boldmath$v$}}(t),\hat{\mbox{\boldmath$v$}}(t)) {\rm d}t=\int_0^{{t^*}} (\partial_t \mbox{\boldmath$\theta$}, \mbox{\boldmath$\rho$} )_{\Omega}{\rm d}t
+(\partial_t \mbox{\boldmath$e$}(0), \hat{\mbox{\boldmath$v$}} (0))_{\Omega}$$
which implies
$$\frac{1}{2} \|\mbox{\boldmath$\rho$}({t^*})\|_{0,\Omega}^2-\frac{1}{2}\|\mbox{\boldmath$\rho$}(0)\|_{0,\Omega}^2+\frac{1}{2}a(\hat{\mbox{\boldmath$v$}}(0),\hat{\mbox{\boldmath$v$}}(0)) =\int_0^{{t^*}} (\partial_t \mbox{\boldmath$\theta$}, \mbox{\boldmath$\rho$})_{\Omega}{\rm d}t
+(\partial_t \mbox{\boldmath$e$}(0), \hat{\mbox{\boldmath$v$}} (0))_{\Omega}.$$
Hence, we deduce
\begin{equation}\label{**} 
\begin{split}
\frac{1}{2} \|\mbox{\boldmath$\rho$}({t^*})\|_{0,\Omega}^2&-\frac{1}{2} \|\mbox{\boldmath$\rho$}(0)\|_{0,\Omega}^2+\frac{1}{2}a(\hat{\mbox{\boldmath$v$}}(0),\hat{\mbox{\boldmath$v$}}(0)) \\
&\leq \max_{0\leq t\leq T} \|\mbox{\boldmath$\rho$}(t)\|_{0,\Omega}\int_0^{{t^*}} \|\partial_t\mbox{\boldmath$\theta$} \|_{0,\Omega}{\rm d}t+\|\partial_t\mbox{\boldmath$e$}(0)\|_{0,\Omega}\| \hat{\mbox{\boldmath$v$}} (0)\|_{0,\Omega}.
\end{split}
\end{equation}
From the Poincar\'e's inequality, the bound of tensor $\mathbb C$, and the Korn's inequality (cf.~\cite{Riv}, ~\cite{Ciarlet}) that
$ \|\nabla \mbox{\boldmath$v$} \|_{0,\Omega}\leq {\sqrt 2}\|\mbox{\boldmath$\varepsilon (v)$}\|_{0,\Omega}$, $\forall {\bm v} \in H^1_0(\Omega)^d$, we get that
\begin{equation}
\begin{split}
\frac{1}{2}a(\hat{\mbox{\boldmath$v$}}(0),\hat{\mbox{\boldmath$v$}}(0))&\geq \frac{1}{2}c_* \|{\bm \varepsilon}(\hat{\mbox{\boldmath$v$}} (0))\|^2_{0,\Omega}\\
&\geq \frac{1}{4}c_* \|\nabla\hat{\mbox{\boldmath$v$}} (0)\|^2_{0,\Omega}
\geq \frac{1}{4}c_* C^{-2}_{\rm F\Omega}\|\hat{\mbox{\boldmath$v$}} (0)\|^2_{0,\Omega},
\end{split}
\end{equation}
Combining
this bound with~(\ref{**}), we arrive at
\begin{equation}
\begin{split}
\frac{1}{2} \|\mbox{\boldmath$\rho$}({t^*})\|_{0,\Omega}^2-\frac{1}{2} \|\mbox{\boldmath$\rho$}(0)\|_{0,\Omega}^2&\leq \max_{0\leq t\leq T} \|\mbox{\boldmath$\rho$}(t)\|_{0,\Omega}\int_0^{{t^*}} \|\partial_t \mbox{\boldmath$\theta$}\|_{0,\Omega}{\rm d}t\\
&+\|\partial_t\mbox{\boldmath$e$}(0)\|_{0,\Omega}\| \hat{\mbox{\boldmath$v$}} (0)\|_{0,\Omega}- \frac{1}{4}c_* C^{-2}_{F\Omega}\|\hat{\mbox{\boldmath$v$}} (0)\|^2_{0,\Omega}.
\end{split}
\end{equation}
Moreover, by estimating the last two terms, we arrive at
\begin{equation}\label{hdt}
\begin{split}
\frac{1}{2} \|\mbox{\boldmath$\rho$}({t^*})\|_{0,\Omega}^2-\frac{1}{2} \|\mbox{\boldmath$\rho$}(0)\|_{0,\Omega}^2
\leq \max_{0\leq t\leq T} \|\mbox{\boldmath$\rho$}(t)\|_{0,\Omega}\int_0^{{t^*}} \|\partial_t\mbox{\boldmath$\theta$} \|_{0,\Omega}{\rm d}t+ {c_*}^{-1} C^2_{F\Omega}\|\partial_t\mbox{\boldmath$e$}(0)\|^2_{0,\Omega},
\end{split}
\end{equation}
Now, we select ${t^*}$ such that $\|\mbox{\boldmath$\rho$}({t^*} )\|_{0,\Omega} = \max_{0 \leq t\leq T}\|\mbox{\boldmath$\rho$}(t)\|_{0,\Omega}$. In~(\ref{hdt}), using the inequality $AB-\frac{1}{4}A^2 \leq B^2$ with $A=\|\mbox{\boldmath$\rho$}({t^*})\|_{0,\Omega}$, $B=\int_0^{{t^*}} \|\partial_t\mbox{\boldmath$\theta$} \|_{0,\Omega}{\rm d}t$ and then taking square root of the resulting inequality yields
\begin{equation}
\begin{split}
\|\mbox{\boldmath$\rho$}\|_{{ L}^\infty (0,T;{ L}^2 (\Omega)^d)} \leq 
\sqrt 2 \|{\bm \rho}(0)\|_{0,\Omega}+ 2 \int_0^T \|\partial_t\mbox{\boldmath$\theta$}\|_{0,\Omega}{\rm d}t +2C_{\rm F\Omega} {c_*}^{-1/2} \|\partial_t{\bm e}(0)\|_{0,\Omega};
\end{split}
\end{equation}

{ Finally we will get the estimate~(\ref{estimatee}) as a consequence of the above estimate}. Indeed, using the bound $\|\mbox{\boldmath$\rho$}(0)\|_{0,\Omega}\leq \|\mbox{\boldmath$e$}(0)\|_{0,\Omega}+\|\mbox{\boldmath$\theta$}(0)\|_{0,\Omega}$; $\mbox{\boldmath$e$}(0)={\bm  u}_h(0)-\mbox{\boldmath$u$}_0$, and $\partial _t\mbox{\boldmath$e$}(0)=\partial_t{\bm  u}_h(0)-\mbox{\boldmath$u$}_1$,
we conclude that
\begin{equation}
\begin{split}
 \|\mbox{\boldmath$e$}\|_{{ L}^\infty (0,T;{ L}^2 (\Omega)^d)} &\leq \|\mbox{\boldmath$\theta$}\|_{{L}^\infty (0,T;{ L}^2 (\Omega)^d)}+\|\mbox{\boldmath$\rho$}\|_{{ L}^\infty (0,T;{ L}^2 (\Omega)^d)}\\
&\leq \|\mbox{\boldmath$\theta$}\|_{{ L}^\infty (0,T;{ L}^2 (\Omega)^d)}+ \sqrt 2 (\|\mbox{\boldmath$u$}_0 - {\bm  u}_h(0)\|_{0,\Omega} +\|\mbox{\boldmath$\theta$}(0) \|_{0,\Omega})\\
&+ 2 \int_0^T \|\partial_t\mbox{\boldmath$\theta$}\|_{0,\Omega}{\rm d}t +2{\rm C_{F\Omega} c_*^{-1/2}} \|\mbox{\boldmath$u$}_1-\partial_t{\bm  u}_h(0)\|_{0,\Omega}.
\end{split}
\end{equation}
%
\end{proof}
 Now it remains to estimate the norms involving the conforming error ${\bm \theta}={\bm w}-{\bm u}_h$ by a computable quantity. According to Remark~\ref{rolew22},   we have that that ${\bm u}_h$ is exact the DG approximation of $\bm w$. We have then the following result.
 \begin{lemma}\label{completing}
Given ${\bm r}\in L^2(\Omega)^d$, consider the stationary elasticity problem of  finding ${\bm z}\in H^1_0(\Omega)^d$ such that
\begin{equation}\label{sta}
-\nabla \cdot ({\bm \sigma}(\bm z))={\bm r}, \quad {\bm z}={\bm 0}~~{\mbox{on}}~\partial \Omega,
\end{equation}
whose solution can be approximated by ${\bm z}_h \in {\bf V}_h$ of the following DG method:
\begin{equation}\label{DGsta}
a_h({\bm z}_h, {\bm v})=( {\bm r}, {\bm v})_\Omega, \quad \mbox{for all } {\bm v} \in {\bf V}_h.
\end{equation}
If we assume that an {\em a posteriori} estimator ${\mathscr E}_{\rm IP}$ exists, i.e.
$$\|{\bm z}-{\bm z}_h\|_{0,\Omega}\leq{\mathscr E}_{\rm IP}({\bm z}_h, {\bm r}, {\mathcal T}_h),$$
then we can  bound the terms in Theorem~\ref{absm} as follows
\begin{subequations}\label{thetaes}
\begin{alignat}{3}
 \|\mbox{\boldmath$\theta$}\|_{{ L}^\infty (0,T;{ L}^2 (\Omega)^d)} &\leq \| {\mathscr E}_{\rm IP}({\bm  u}_h,{ \bm g}, {\mathcal T}_h)\|_{L^\infty(0,T)};\\
\sqrt{2}\|\mbox{\boldmath$\theta$}(0)\|_{0,\Omega}&\leq \sqrt{2} {\mathscr E}_{\rm IP}({\bm  u}_h(0),{ \bm g}(0), {\mathcal T}_h);\\
 2 \int_0^T \|\partial_t\mbox{\boldmath$\theta$}\|_{0,\Omega}{\rm d}t &\leq 2\int_0^T  {\mathscr E}_{\rm IP}({\partial_t\bm  u}_h,\partial_t{ \bm g}, {\mathcal T}_h ){\rm d}t;
\end{alignat}
\end{subequations}
where ${\bm g}={{\bm B} {\bm u}_h}-{\sf \Pi}_h{\bm f}+{\bm f}$, and under the assumption that $\bm f$ is differentiable in time.
\end{lemma}
\begin{proof}
The first two bounds of~(\ref{thetaes}) are directly derived from the property of $\bm w$ as presented in Remark~\ref{rolew22}. For the third estimate, noting that  $a(\cdot,\cdot)$ and $a_h(\cdot,\cdot)$ are independent of $t$, then there holds
$$a({\partial_t\bm  w}, {\bm v})=({\partial_t\bm  g}, {\bm v})_\Omega, ~\mbox{for all}~~{\bm v} \in H^1_0(\Omega)^d,$$
and $$a_h({\partial_t\bm  u}_h, {\bm v})=(\partial_t ({\bm B}{\bm u}_h), {\bm v})_\Omega=({\partial_t\bm  g}, {\bm v})_\Omega, ~\mbox{for all}~~{\bm v} \in {\bf V}_h,$$
noting that here we make use of the fact that the projection ${\sf \Pi}_h$ commutes with the time differentiation.
The estimate  is then the difference between ${\partial_t \bm  u}_h$ and its reconstruction $\partial_t {\bm w}$.
\end{proof}

\section{{ A posteriori} residual bounds for the stationary elasticity problem}
In this section, we present two methods to derive  {\em a posteriori} error bounds in $L^2$-norm  for the stationary elasticity problem. The first one is based on the method of duality, and the second one is based on rewriting the DG formulation with lifting operator to derive the error estimate in energy norm which will provide a bound in the $L^2$-norm of the error as well.
 

\subsection{{\em A posteriori} error bound in $L^2$-norm  by duality method}
\begin{theorem}\label{IP2} {\rm ({Error obtained by duality method})}.
 Let ${\bm z} \in { H}^1_{0}(\Omega)^d$ be the solution to the stationary  problem~(\ref{sta}) and ${{\bm z}_h} \in {\bf V}_h$ be the DG approximation of $\mbox{\boldmath$z$}$  as in ~(\ref{DGsta}). Then the error for the DG method is estimated by
$$\|\mbox{\boldmath$z$}-{{\bm z}_h}\|_{0,\Omega} \leq 
 \mathscr{E}_{\rm IP}({{\bm z}_h}, \mbox{\boldmath$r$},{\mathcal T}_h),$$
with $\mathscr{E}_{\rm IP}$ is given by
\begin{align*}
\mathscr{E}_{\rm IP}({{\bm z}_h}, \mbox{\boldmath$r$}, {\mathcal T}_h ) :=C\bigg \{\sum_{K\in {\mathcal T}_h} h_K^4 &\| {\bm r}+{\nabla \cdot ({\bm \sigma ({\bm z}_h)})}\|^2_{0,K} \\
&+\sum_{e\in {\mathcal E}_h^I} {{\sf h}}^3 \|\llbracket{\bm \sigma} ({\bm z}_h) {\bm \nu}_e\rrbracket\|_{0,e}^2+\sum_{e\in {\mathcal E}_h} {{\sf h}} \|\llbracket {\bm z}_h\rrbracket\|^2_{0,e}\bigg\}^{1/2},
\end{align*}
where $C$ is a positive constant independent of ${{\bm z}_h}, {\bm r}, h$ and ${\mathcal T}_h$.
\end{theorem}
\begin{proof}
In this proof we denote by ${\sf e}={\bm z}-{\bm z}_h$. We consider  the adjoint  problem
\begin{equation}\label{adjointpr}
\begin{split}
-\nabla \cdot ({\bm \sigma ({\bm \phi})})&={\sf e}, \quad \mbox{in}~ \Omega\\
{\bm \phi}&={\bm 0},\quad \mbox{on} ~\partial \Omega,
\end{split}
\end{equation}
and assume that this problem is regular in the sense that the solution ${\bm \phi}$ satisfies    ${\bm \phi} \in H^2(\Omega)^d$ with continuous dependence on ${\sf e}$:
\begin{equation}\label{regulardual}
\|{\bm \phi}\|_{2,\Omega} \leq C \|{\sf e}\|_{0,\Omega}.
\end{equation}
This assumption is known to hold, in particular, if the domain $\Omega$ is convex (cf.~\cite{Ciarlet}).
On the other hand, along with Eq.~(\ref{adjointpr}), we get the following expression
\begin{equation}
\|{\sf e}\|^2_{0,\Omega}=\sum_{K\in {\mathcal T}_h}\int_{K} (-\nabla\cdot ({\bm \sigma ({\bm \phi})}))\cdot {\sf e}{\rm d}{\bm x}.
\end{equation}
Now applying integration by parts on each element and using the fact that the term $\llbracket{{(\bm \sigma ({\bm \phi}) ){\bm \nu}_e}}\rrbracket={\bm 0}$ a.e. on $e$, $\forall e\in {\mathcal E}_h^I$, we obtain
\begin{equation}\label{E2sub}
\|{\sf e}\|^2_{0,\Omega}=\sum_{K\in {\mathcal T}_h}\int_{K}  ({\bm \sigma ({\bm \phi})}): {\bm \varepsilon}({\sf e}) {\rm d}{\bm x} -\sum_{e \in {\mathcal E}_h}\int_e  \llrrbrace{({\bm \sigma ({\bm \phi})}){\bm \nu}_e}\cdot \llbracket{{\sf e}}\rrbracket {\rm d} A.
\end{equation}

Let ${\bm \phi}^*$ be a continuous interpolant  of $\bm \phi$  as in Lemma~\ref{bestapp2} satisfying  properties in~(\ref{appv3}), note that with this interpolation we  also have that  ${\bm \phi}^*={\bm \phi}={\bm 0}$ on $\partial \Omega$ (see~\cite{R-W-G}). Thanks to the consistency of the bilinear form $a_h$, we get the orthogonality equation $a_h({\sf e},{\bm \phi}^*)=0$. Then the following expression holds
\begin{equation}
\begin{split}
0= \sum_{K\in {\mathcal T}_h}\int_{K}  {\bm \sigma ({\bm \phi}^*)}: {\bm \varepsilon}({\sf e}) {\rm d}{\bm x} -\sum_{e \in {\mathcal E}_h}\int_e \llrrbrace{({\bm \sigma ({\bm \phi}^*)}){\bm \nu}_e}\cdot \llbracket{{\sf e}}\rrbracket {\rm d} A,
\end{split}
\end{equation}
noting that the jump terms $\sum_{e\in {\mathcal E}_h} \int_e \llrrbrace{({\bm \sigma ({\sf e})}){\bm \nu}_e}\cdot \llbracket{{\bm \phi} ^*}\rrbracket {\rm d}A$ and $\sum_{e\in {\mathcal E}_h} \int_e \llbracket{{\bm  \phi}^*}\rrbracket\cdot \llbracket{{\sf e}}\rrbracket {\rm d}A$ vanish due to the continuity of ${\bm \phi}^*$ which is  followed by our choice.

Subtracting the above equation from~(\ref{E2sub}), by noticing that $\int_{K}  ({\bm \sigma} ({\bm \phi})): {\bm \varepsilon}({\sf e}) {\rm d}{\bm x}=\int_{K}  ({\bm \sigma} ({\sf e})): {\bm \varepsilon}({\bm \phi}) {\rm d}{\bm x}$  we have the following equation
\begin{equation}\label{E2sub2}
\|{\sf e}\|^2_{0,\Omega}=\sum_{K\in {\mathcal T}_h}\int_{K}  {\bm \sigma ({\sf e})}: {\bm \varepsilon}({\bm \phi}-{\bm \phi}^*) {\rm d}{\bm x}-\sum_{e \in {\mathcal E}_h}\int_e \llrrbrace{({\bm \sigma ({\bm \phi}-{\bm \phi}^*)}){\bm \nu}_e}\cdot \llbracket{{\sf e}}\rrbracket {\rm d} A.
\end{equation}
Integrate by parts the first term from ~(\ref{E2sub2}),  we obtain
\begin{equation*}
\begin{split}
\|{\sf e}\|^2_{0,\Omega}&=\sum_{K\in {\mathcal T}_h}\int_{K}  \big({\bm f}+{\nabla \cdot ({\bm \sigma}({\bm z}_h)})\big)\cdot ({\bm \phi}-{\bm \phi}^*) {\rm d}{\bm x} - \sum_{e \in {\mathcal E}_h^I}\int_e  \llbracket{ ({\bm \sigma}({\bm z}_h)){\bm \nu}_e}\rrbracket \cdot \llrrbrace{ {\bm \phi}-{\bm \phi}^*}{\rm d} A\\
&+\sum_{e \in {\mathcal E}_h}\int_e \llrrbrace{({\bm \sigma ({\bm \phi}-{\bm \phi}^*)}){\bm \nu}_e}\cdot \llbracket{{\bm z}_h}\rrbracket {\rm d} A.
\end{split}
\end{equation*} 
Then by employing the Cauchy-Schwarz's inequality, we have that
\begin{equation}\label{E2sub3}
\begin{split}
\|{\sf e}\|^2_{0,\Omega}&\leq  \big(\sum_{K\in {\mathcal T}_h} h_K^4\|{\bm f}+{\nabla \cdot ({\bm \sigma ({\bm z}_h)})}\|^2_{0,K}\big)^{1/2}\big (\sum_{K\in {\mathcal T}_h}h_K^{-4} \|{\bm \phi}-{\bm \phi}^*\|^2_{0,K}\big)^{1/2} \\
&+\big( \sum_{e \in {\mathcal E}_h^I} {\sf h}^3\|\llbracket{ ({\bm \sigma}({\bm z}_h)){\bm \nu}_e}\rrbracket \|_{0,e}^2\big)^{1/2}\big(\sum_{e \in {\mathcal E}_h^I} {\sf h}^{-3}\| {\bm \phi}-{\bm \phi}^*\|^2_{0,e}\big)^{1/2}\\
&+\big (\sum_{e \in {\mathcal E}_h} {\sf h}^{-1}\|\llrrbrace {({\bm \sigma ({\bm \phi}-{\bm \phi}^*)}){\bm \nu}_e}\|^2_{0,e}\big)^{1/2} \big(\sum_{e \in {\mathcal E}_h}{\sf h}\|\llbracket{{\bm z}_h}\rrbracket\|^2_{0,e}\big)^{1/2}.
\end{split}
\end{equation}
Now applying the properties of the interpolation ${\bm \phi}^*$ as in ~(\ref{bestapp2}), we have the following estimates for an element $K$ of the triangulation and an edge or face  $e$  on $\partial K$:
\begin{equation}
\begin{split}
\|{\bm \phi}-{\bm \phi}^*\|_{0,K}^2 &\leq C h_K^4 \|{\bm \phi}\|^2_{2,K}\\
\|{\bm \phi}-{\bm \phi}^*\|_{0,e}^2 &\leq C h_K^3 \|{\bm \phi}\|^2_{2,K}\\
\|({\bm \sigma}({\bm \phi}-{\bm \phi}^*)){\bm \nu}_e\|_{0,e}^2 &\leq C h_K\|{\bm \phi}\|^2_{2,K}.
\end{split}
\end{equation}
 Finally, applying the above approximation results  to ~(\ref{E2sub3}), we deduce
\begin{equation}
\begin{split}
\|{\sf e}\|_{0,\Omega}^2 \leq C \bigg( &\sum_{K\in} h_K^4 \|{\bm f}+\nabla \cdot {\bm \sigma }({\bm z}_h)\|^2_{0,K}\\
&+{\sf h}^3\sum_{e\in {\mathcal E}_h^I}\|\llbracket{({\bm \sigma}({\bm z}_h)){\bm \nu}_e}\rrbracket\|^2_{0,e}+\sum_{e\in {\mathcal E}_h} {\sf h}\|\llbracket{{\bm z}_h}\rrbracket\|^2_{0,e}
\bigg)^{1/2}\|{{\bm \phi}}\|_{2,\Omega},
\end{split}
\end{equation}
noting that in the last two inequalities, we make use of the locally quasi-uniform of the triangulation.

Then making use of the regularity  assumption~(\ref{regulardual}) of the dual problem which allows to bound the norm of $\bm \phi$ by the norm of ${\sf e}$, the proof of the theorem is completed.
\end{proof}

\subsection{{\em A posteriori} error bound in $L^2$-norm  {\em via} energy norm}
\subsubsection{Extension of the bilinear  form $a_h$ on larger space}\label{extension}
To state our {\em a posteriori} error bounds, we define the space
${\bf V}(h) = H_0^1(\Omega)^d+ {\bf V}_h$.
On ${\bf V}(h)$, we define the mesh-dependent energy norm for ${{\bm v}}\in{\bf V}(h) $:
\begin{equation}\label{DGnormex}
\vertiii{\bm {v}}=\big(\sum_{K\in {\mathcal{T}}_h}\|{\bm \varepsilon}({\bm v})\|^2_{0,K}+\sum_{e\in {\mathcal{E}}_h} {\texttt a} \|\llbracket{\mbox{\boldmath$v$}}\rrbracket\|^2_{0,e}\big)^{1/2}.
\end{equation}
We shall extend the DG form $a_h$ in~(\ref{ahuv}) to the space ${\bf V}(h)\times {\bf V}(h)$ by using the lifting operator. This way of expending the bilinear form can be seen from ~\cite{G-S-S}.
\subsubsection{Lifting operator}
We denote by $\bm {\Sigma}_h$ the following space:
$${\bm \Sigma}_h=\{\mbox{\boldmath$\tau$} \in L^2(\Omega)^{d\times d}: \mbox{\boldmath$\tau$} |_K \in {\mathbb P}_r(K)^{d\times d}, K \in {\mathcal T}_h\}.$$
Observing that $\mbox{\boldmath$\varepsilon$}({\bf V}_h) \subset {\bm \Sigma}_h$, and for the piecewise constant tensor ${\mathbb C}$, we also have ${\mathbb C}\mbox{\boldmath$\varepsilon$}({\bf V}_h) \subset {\bm \Sigma}_h$.
\begin{definition}\label{liftingdef} For ${{\bm v}}\in {\bf V}(h) $, the lifting operator ${\mathscr L}: {\bf V}(h) \rightarrow {\bm \Sigma}_h$ is defined by
$$\int_{\Omega} {\mathscr L}(\mbox{\boldmath$v$}):\mbox{\boldmath$\tau$}{\rm d}{\bm x}=\sum_{e\in{\mathcal E}_h}\int_{e} \llbracket{\mbox{\boldmath$v$}}\rrbracket\cdot \llrrbrace{({\mathbb C}\mbox{\boldmath$\tau$} ){\bm \nu}_e}{\rm d}A,~~~\forall \mbox{\boldmath$\tau$} \in {\bm \Sigma}_h.$$
\end{definition}
Note that ${\mathscr {L}}$ is well-defined, according to  the   Riesz representation theorem for the $L^2$ scalar product in ${\bm \Sigma}_h$ (see e.g.~\cite{G-S-S}, ~\cite{Wihler}). 

\begin{lemma}{\em (Stability of lifting operator)}.\label{liftingstable}
There exists a constant $C>0$ independent of the mesh size such that
 $$\|{\mathscr{L}}(\mbox{\boldmath$v$})\|_{0,\Omega} \leq C_{\rm{inv}}C^* \big ( \sum_{e \in {\mathcal E}_h}{\sf h}^{-1} \|\llbracket{\mbox{\boldmath$v$}}\rrbracket\|_{0,e}^2\big )^{1/2},~~~\mathrm {for~ any}~ {{\bm v}}\in {\bf V}(h).$$
\end{lemma}
\begin{proof}
By the definition of the operator ${\mathscr L}$ and from Cauchy-Schwarz's inequality we have for any ${{\bm v}}\in {\bf V}(h)$:
\begin{equation}
\begin{split}
\|{\mathscr L}(\mbox{\boldmath$v$})\|_{0,\Omega}&=\sup_{{\bm z} \in {\bf V}_h}\frac{\int_\Omega {\mathscr L}({\bm v})\cdot {\bm z}{\rm d}{\bm x}}{\|\bm z\|_{0,\Omega}}\\
&=\sup_{{\bm z} \in {\bf V}_h}\frac{\sum_{e\in {\mathcal E}_h}\int_e \llbracket{\mbox{\boldmath$v$}}\rrbracket\cdot \llrrbrace{({\mathbb C} {\bm z}) {\bm \nu}_e}{\rm d}A}{\|\bm z\|_{0,\Omega}}\\
&\leq \sup_{{\bm z} \in {\bf V}_h} \frac{\left( \sum_{e \in {\mathcal E}_h}{\sf h}^{-1} \|\llbracket{\mbox{\boldmath$v$}}\rrbracket\|_{0,e}^2\right)^{1/2}\left(  \sum_{e \in {\mathcal E}_h}{\sf h} \|\llrrbrace{({\mathbb C} {\bm z}) {\bm \nu}_e}\|_{0,e}^2\right)^{1/2}}{\|\bm z\|_{0,\Omega}},\\
& \leq C_{\rm{inv}}C^* \big ( \sum_{e \in {\mathcal E}_h}{\sf h}^{-1} \|\llbracket{\mbox{\boldmath$v$}}\rrbracket\|_{0,e}^2\big )^{1/2};
\end{split}
\end{equation}
noting that in  the above proof, we have used the inequalities
 $${\sf h}^{1/2}\|\llrrbrace{({\mathbb C} {\bm z}) {\bm \nu}_e}\|_{0,e}\leq \frac{1}{2}C_{\rm{inv}}\left (\|{\mathbb C} {\bm z}\|^2_{0, K^+}+\|{\mathbb C} {\bm z}\|^2_{0, K^-}\right)^{1/2},$$
coming directly from  the inverse inequality~~(\ref{invconstant}) and the locally quasi-uniform  of the mesh.
 \end{proof}
With these definitions, we can now introduce the following (non-consistent) DG form on ${\bf V}(h) \times {\bf V}(h)$:
\begin{equation}\label{newDG}
\begin{split}
 A(\mbox{\boldmath$u$},\mbox{\boldmath$v$} )&=\sum_{K\in {\mathcal{T}}_h} \int_{K}{\bm \sigma}({{\bm u}}):{\bm \varepsilon}({{\bm v}}){\rm d}{\bm x}-\sum_{K\in {\mathcal{T}}_h} \int_{K}{\bm \varepsilon} ({{\bm u}}):{\mathscr{L}}(\mbox{\boldmath$v$}){\rm d}{\bm x}\\
& -\sum_{K\in {\mathcal{T}}_h} \int_{K}{\bm \varepsilon}({{\bm v}}):{\mathscr{L}}(\mbox{\boldmath$u$}){\rm d}{\bm x}+\sum_{e\in \mathcal{E}_h} \int_{e}{\texttt a} \llbracket{\mbox{\boldmath$u$}}\rrbracket\cdot\llbracket{\mbox{\boldmath$v$}}\rrbracket{\rm d}A;
\end{split}
\end{equation}
where $\alpha$ is a positive constant which we will consider later on to establish the coercivity of the bilinear form $A(\cdot,\cdot)$.\\

The bilinear form $A$ above  is continuous and coercive on the entire space
${\bf V} (h) \times {\bf V} (h)$.
  Setting
$\alpha_{\rm {min}} = 4 C_{\textrm{inv}}^2(C^*)^2c_*^{-1}$.
There are constants $M$  and $\kappa$ independent of $h$ such  that
$$|A ({{\bm u}},{{\bm v}})| \leq M \vertiii{{\bm u}} \cdot \vertiii{{\bm v}},~~~\forall {{\bm u}},{{\bm v}} \in {\bf V} (h),$$
and for $\alpha \geq\alpha_{\rm {min}} $:
$$A ({{\bm u}},{{\bm u}}) \geq \kappa \vertiii{{\bm u}}^2,~~~\forall {{\bm u}} \in {\bf V} (h),$$
with DG norm $\vertiii{\cdot}$ defined in~(\ref{DGnormex}).

 Furthermore, since
$A = a_h$
on ${\bf V}_h \times {\bf V}_ h$,
and 
$A = a$ on $H^1_0(\Omega)^d\times H^1_0(\Omega)^d$,
the form $A$ can be viewed as an extension of the two forms $a_h$ and $a$ to the space
${\bf V} (h) \times {\bf V} (h)$.

The discrete problem then can be equivalently stated as follows: 
\begin{lemma}(\em Non-consistent DG formulation).
The DG formulation ~(\ref{DGsta}) is equivalent to finding  ${{\bm z}_h} \in  {\bf V}_h$ such that 
\begin{equation}\label{DGAbi}
 A({{\bm z}_h},{\bm  v} ) =( {\bm r}, {{\bm v}})_\Omega~~\mbox{ for all}~ {{\bm v}} \in {\bf V}_h.
\end{equation}
\end{lemma}
\subsubsection{{ A posteriori} error bound}
The proof of the {\em a posteriori} error bound is based on the approximation results in~\cite{K-P} which allow us to find a conforming finite element function which is close to any discontinuous one. 
\begin{lemma}\label{Ka}
 ({\rm Bounding the nonconforming part {\em via} jumps, cf.~\cite{K-P}}). 
 Let ${\bf V}_h$ be the space of piecewise polynomial functions as defined in ~(\ref{spacevh}), then for any
function ${\bm  z}_h \in {\bf V}_h$ there exists a function ${\bm  z}_h^c \in  {\bf V}_c={\bf V}_h \cap { H}^1_0(\Omega)^d$ such that
$$\|{{\bm z}_h} - {\bm  z}_h^c\|^2_{0,\Omega} \leq C_1\big (\sum_{e\in {\mathcal{E}}_h}{ \sf h} \|\llbracket{{\bm z}_h}\rrbracket\|^2_{0,e }\big),$$
and
$$\sum_{K\in {\mathcal T}_h }\|\nabla  ({{\bm z}_h} - {\bm  z}_h ^c)\|^2_{0,K} \leq C_2 \big(\sum_{e\in {\mathcal{E}}_h}{\sf h}^{-1}\| \llbracket{{\bm z}_h}\rrbracket\|^2_{0,e}\big),$$
where $C_1 , C_2 > 0$ constants depending on the shape-regularity of the triangulation.
\end{lemma}
\begin{theorem}\label{IP}{\rm ($L^2-$}a posteriori {\rm error bound via energy norm)}.
 Let ${\bm z} \in { H}^1_{0}(\Omega)^d$ be the solution to the stationary elasticity problem~(\ref{sta}) and ${{\bm z}_h} \in {\bf V}_h$ be the DG approximation of $\mbox{\boldmath$z$}$  as in~(\ref{DGAbi}). Then
$$\|\mbox{\boldmath$z$}-{{\bm z}_h}\|_{0,\Omega} \leq 
 \mathscr{E}_{\rm IP}({{\bm z}_h}, \mbox{\boldmath$r$},{\mathcal T}_h),$$
where 
\begin{equation*}
\begin{split}
\mathscr{E}_{\rm IP}({{\bm z}_h}, \mbox{\boldmath$r$}, {\mathcal T}_h ) :=C \bigg( \sum_{K \in {\mathcal{T}}_h}{h}_K^2 \|\mbox{\boldmath$r$}+\nabla\cdot \mbox{\boldmath$\sigma$}({{\bm z}_h})\|^2_{0,K}&+ \sum_{e\in {\mathcal E}^I_{h}}{\sf h}\|\llbracket{({\bm \sigma}({{\bm z}_h})) {\bm \nu}_e}\rrbracket\|^2_{0,e}\\
&+\sum_{e\in {\mathcal{E}}_h}({\sf h}+\alpha {\sf h}^{-1})\|\llbracket{{\bm z}_h}\rrbracket\|^2_{0,e} \bigg)^{1/2},
\end{split}
\end{equation*}
where $C$ is a positive constant independent of ${{\bm z}_h}, {\bm r}, h$ and ${\mathcal T}_h$.
\end{theorem}

\begin{proof}
We decompose the error as follows
$${\bm z}- {{\bm z}_h}=\underbrace{{\bm z}- {\bm  z}_h^c}_{{\bm e}_c  \in {H}^1_0(\Omega)^d}+\underbrace{{\bm  z}_h^c- {\bm z}_h}_{{\bm e}_d\in {\bf V}_h}  ,$$
noting that ${\bm e}_c  \in {H}^1_0(\Omega)^d$ and ${\bm e}_d\in {\bf V}_h  $.
where ${\bm z}_h^c \in {\bf V}_c$ is the conforming approximation of ${\bm z}_h$ from Lemma~\ref{Ka}.

 Let ${\tilde{\bm e}}_c$ denote an  approximation of ${\bm e}_c$ in the space of element-wise constant vector functions as in Lemma~\ref{bestapp}, we have that ${\tilde{\bm e}}_c\in {\bf V}_h $. Define $\mbox{\boldmath$\eta$}=\mbox{\boldmath$e$}_c-{\tilde{\bm e}}_c$, we then have that
 $$A({\bm e},{\bm e}_c)=A({\bm z},{\bm e}_c)-A({{\bm z}_h},\mbox{\boldmath$e$}_c)=( \mbox{\boldmath$r$},\mbox{\boldmath$e$}_c )_{\Omega}-A({{\bm z}_h},\mbox{\boldmath$\eta$})-A({{\bm z}_h},{\tilde{\bm e}}_c).$$
 $$=( \mbox{\boldmath$r$},\mbox{\boldmath$\eta$})_{\Omega}-A({{\bm z}_h},\mbox{\boldmath$\eta$}).$$
 Moreover, noting that $A({{\bm z}_h},{\tilde{\bm e}}_c)=( \mbox{\boldmath$r$},{\tilde{\bm e}}_c)_{\Omega}$,
 then we have
 \begin{equation}\label{ec}
c_* \|{\bm \varepsilon}(\mbox{\boldmath$e$}_c)\|^2_{0,\Omega}\leq A(\mbox{\boldmath$e$}_c,\mbox{\boldmath$e$}_c)=( \mbox{\boldmath$r$},\mbox{\boldmath$\eta$})_{\Omega}-A({{\bm z}_h},\mbox{\boldmath$\eta$})-A(\mbox{\boldmath$e$}_d,\mbox{\boldmath$e$}_c).
 \end{equation}
 {\em Now our purpose is to bound the terms on the right hand side of~(\ref{ec}) }.
 \begin{enumerate}[i)]
 \item
Making use of integrating by parts element-wise $A({{\bm z}_h},\mbox{\boldmath$\eta$})$, then the sum of the first two terms on the right hand side of~(\ref{ec})  can be written as follows:
 \begin{equation}\label{2term}
 \begin{split}
 ( \mbox{\boldmath$r$},\mbox{\boldmath$\eta$})_\Omega-A({{\bm z}_h},\mbox{\boldmath$\eta$})=\sum_{K \in {\mathcal T}_h} \int_K  \left( \mbox{\boldmath$r$}+ \nabla \cdot \mbox{\boldmath$\sigma$}({{\bm z}_h})\right)\cdot\mbox{\boldmath$\eta$}{\mathrm d}{\bm x}-\sum_{e \in  {\mathcal E}_h^I}\int_e\llrrbrace{\mbox{\boldmath$\eta$}}\cdot \llbracket(\mbox{\boldmath$\sigma$}({{\bm z}_h}) ) {\bm \nu}_{e}\rrbracket{\rm d}A~\hspace{1cm}\\
 +\sum_{K \in {\mathcal T}_h} \int_K {\mathscr L}({{\bm z}_h}):\mbox{\boldmath$\varepsilon$}(\mbox{\boldmath$\eta$}){\mathrm d}{\bm x} -\sum_{e\in {\mathcal E}_h}\int_e\frac{\alpha}{{\sf h}}\llbracket\mbox{\boldmath$\eta$}\rrbracket\cdot \llbracket{{\bm z}_h}\rrbracket{\rm d}A~\hspace{2cm}.
  \end{split}
 \end{equation}
  Now we work on the parts on the right hand-side of~(\ref{2term}).\\
  {\em  For the first part of~(\ref{2term})}, by employing the discrete Cauchy-Schwarz's inequality  we have
  \begin{equation}\label{term2a}
   \begin{split}
\big|\sum_{K \in {\mathcal T}_h} \int_K ({\bm r} &+\nabla \cdot \mbox{\boldmath$\sigma$}({{\bm z}_h}))\cdot \mbox{\boldmath$\eta$}{\mathrm d}{\bm x}\big| \\
&\leq  \big(\sum_{{K\in {\mathcal T}_h}}h_K^2\|{\bm r} +\nabla \cdot \mbox{\boldmath$\sigma$}({{\bm z}_h})\|^2_{0,K}\big)^{1/2}\big(\sum_{K\in {\mathcal T}_h}h_K^{-2}\|\mbox{\boldmath$\eta$}\|^2_{0,K}\big)^{1/2}\\
  &\leq C\big( \sum_{K\in {\mathcal T}_h}h_K^2\| {\bm r} +\nabla \cdot \mbox{\boldmath$\sigma$}({{\bm z}_h})\|^2_{0,K}\big)^{1/2}\|\nabla \mbox{\boldmath$e$}_c\|_{0,\Omega},
   \end{split}
  \end{equation} 
  noting the last inequality comes from the approximation~(\ref{bestapp}).\\
 {\em For the second part of~(\ref{2term})}
\begin{equation}\label{term2b}
\begin{split}
\big |\sum_{e \in  {\mathcal E}^I_h}\int_{e}\llrrbrace{\mbox{\boldmath$\eta$}}\cdot &\llbracket(\mbox{\boldmath$\sigma$}({{\bm z}_h}) ){\bm \nu}_{e}\rrbracket{\rm d}A\big|\\
&\leq \frac{1}{2}\big(\sum_{K \in {\mathcal T}_h}{\sf h }^{-1}\|\mbox{\boldmath$\eta$}\|^2_{0,\partial K}\big)^{1/2}\big(\sum_{e \in  {\mathcal E}_h^I}{\sf h }\|\llbracket(\mbox{\boldmath$\sigma$}({{\bm z}_h}) ) \mbox{\boldmath$\nu$}_{e}\rrbracket\|^2_{0,e}\big)^{1/2}.
\end{split}
\end{equation}
Now to estimate the first factor of~(\ref{term2b}), we will use the trace theorem~(\ref{tr}) in conjunction with the  approximation property~(\ref{bestapp}):
\begin{equation}
\begin{split}
{\sf h }^{-1}\|\mbox{\boldmath$\eta$}\|^2_{0,\partial K}&\leq C_{\textrm{tr}}{\sf h }^{-1}\left( h_K^{-1}\|\bm \eta\|_{0,K}^2+h_K\|\nabla {\bm \eta}\|_{0,K}^2\right)\\
&\leq C \left ( h_K^{-2} \|\bm \eta\|_{0,K}^2 +\|\nabla {\bm e}_c\|_{0,K}^2\right)\\
&\leq C \|\nabla {\bm e}_c\|_{0,K}^2;
\end{split}
 \end{equation}
 noting that in the last two inequalities, we make use of the fact that $\nabla \mbox{\boldmath$\eta$}|_K=\nabla \mbox{\boldmath$e$}_c|_K$ for all $K\in {\mathcal T}_h$, the locally quasi-uniform property  of the mesh
and the inverse inequality~(\ref{inv}), then  ~(\ref{term2b}) infers that
 \begin{equation}
 \begin{split}
 \big|\sum_{e \in  {\mathcal E}^I_h}\int_{e}\llrrbrace{\mbox{\boldmath$\eta$}}\cdot\llbracket(\mbox{\boldmath$\sigma$}({{\bm z}_h}) ) {\bm \nu}_{e}\rrbracket{\rm d}A\big|
 &\leq C \|\nabla {\bm e}_c\|_{0,K} \times \big (\sum_{e \in  {\mathcal E}_h^I}{\sf h }\|\llbracket(\mbox{\boldmath$\sigma$}({{\bm z}_h}))\mbox{\boldmath$\nu$}_{e}\rrbracket\|^2_{0,e}\big)^{1/2}\\
 &\leq C \|\nabla \mbox{\boldmath$e$}_c\|_{0,\Omega}\big(\sum_{e \in  {\mathcal E}_h^I}{\sf h}\|\llbracket(\mbox{\boldmath$\sigma$}({{\bm z}_h})) \mbox{\boldmath$\nu$}_{e}\rrbracket\|^2_{0,e}\big)^{1/2}.
 \end{split}
\end{equation}
{\em For the third part of~(\ref{2term})}
\begin{equation}\label{term2c}
\begin{split}
\big|\sum_{K \in {\mathcal T}_h} \int_K {\mathscr L}({{\bm z}_h}):\mbox{\boldmath$\varepsilon$}(\mbox{\boldmath$\eta$}){\mathrm d}{\bm x} \big|
&\leq \big(\sum_{K \in {\mathcal T}_h}\|{\mathscr L}({{\bm z}_h})\|^2_{0,K}\big)^{1/2}\big(\sum_{K \in {\mathcal T}_h}\|\mbox{\boldmath$\varepsilon$}(\mbox{\boldmath$\eta$})\|^2_{0,K}\big)^{1/2}\\
&\leq C \big(\sum_{e\in \mathcal{E}_h}{\sf h }^{-1}\|\llbracket{{\bm z}_h}\rrbracket\|^2_{0,e}\big)^{1/2}\big(\sum_{K \in {\mathcal T}_h}\|\mbox{\boldmath$\varepsilon$}(\mbox{\boldmath$\eta$})\|^2_{0,K}\big)^{1/2}\\
&\leq C \big(\sum_{e\in {\mathcal E}_h}{\sf h }^{-1}\|\llbracket{{\bm z}_h}\rrbracket\|^2_{0,e}\big)^{1/2}\|\mbox{\boldmath$\varepsilon$}(\mbox{\boldmath$e$}_c)\|_{0,\Omega} ;
\end{split}
\end{equation}
noting that $\mbox{\boldmath$\varepsilon$}(\mbox{\boldmath$\eta$})|_K=\mbox{\boldmath$\varepsilon$}(\mbox{\boldmath$e$}_c)|_K$ for all $K\in {\mathcal T}_h$.\\
 {\em And the final part of~(\ref{2term})}:
 \begin{equation}\label{term2d}
  \begin{split}
 |\sum_{e\in {\mathcal E}_h}\int_e\frac{\alpha}{{\sf h }}\llbracket\mbox{\boldmath$\eta$}\rrbracket\cdot \llbracket{{\bm z}_h}\rrbracket{\rm d}A|& \leq C (\sum_{e\in {\mathcal E}_h}{\sf h }^{-1}\|\llbracket\mbox{\boldmath$\eta$}\rrbracket\|^2_{0,e})^{1/2} \times (\sum_{e\in {\mathcal E}_h}{\sf h }^{-1}\|\llbracket{{\bm z}_h}\rrbracket\|^2_{0,e})^{1/2}\\
 &\leq C\|\nabla \mbox{\boldmath$e$}_c\|_{0,\Omega}\times (\sum_{e\in {\mathcal E}_h}{\sf h }^{-1}\|\llbracket{{\bm z}_h}\rrbracket\| ^2_{0,e})^{1/2},
  \end{split}
 \end{equation}
 noting that here we use the same approach as in~(\ref{term2b}).
 \item 
It remains now to estimate   the   term $-A(\mbox{\boldmath$e$}_d,\mbox{\boldmath$e$}_c)$ from~(\ref{ec}). We have that
 \begin{equation}
 \begin{split}
 -A(\mbox{\boldmath$e$}_d,\mbox{\boldmath$e$}_c)&\leq C^* \big(\sum_{K \in {\mathcal T}_h}\|\mbox{\boldmath$\varepsilon$}(\mbox{\boldmath$e$}_d)\|_{0,K}\big)\|\mbox{\boldmath$\varepsilon$}(\mbox{\boldmath$e$}_c)\|_{0,\Omega}+\|{\mathscr L}({\mbox{\boldmath$e$}_d})\|_{0,\Omega}\|\mbox{\boldmath$\varepsilon$}(\mbox{\boldmath$e$}_c)\|_{0}\\
& \leq C \|\mbox{\boldmath$\varepsilon$}(\mbox{\boldmath$e$}_c)\|_{0,\Omega}\left(\sum_{K \in {\mathcal T}_h} \|\mbox{\boldmath$\varepsilon$}(\mbox{\boldmath$e$}_d)\|^2_{0,K}+ \sum_{e\in \mathcal{E}_h}{\sf h }^{-1}\|\llbracket{\mbox{\boldmath$e$}_d}\rrbracket\|^2_{0,e}\right)^{1/2}.
 \end{split}
 \end{equation}
 Observing that $\|\llbracket{\mbox{\boldmath$e$}_d}\rrbracket\|_e=\|\llbracket{\bm z}_h\rrbracket\|_e, \forall {e\in \mathcal{E}_h} $. We also have $\|\mbox{\boldmath$\varepsilon$}(\mbox{\boldmath$e$}_d)\|_{0,K} \leq C\|\nabla \mbox{\boldmath$e$}_d\|_{0,K}$ for $K\in {\mathcal T}_h$, and remark that from Lemma~\ref{Ka}, the nonconforming  term can be bounded in terms of the jumps of discrete solution. We then obtain the following result
  \begin{equation}\label{term3}
  -A(\mbox{\boldmath$e$}_d,\mbox{\boldmath$e$}_c)\leq C \|\mbox{\boldmath$\varepsilon$}(\mbox{\boldmath$e$}_c)\|_{0,\Omega}(\sum_{e\in \mathcal{E}_h}{\sf h }^{-1}\|\llbracket{{\bm z}_h}\rrbracket\|^2_{0,e})^{1/2},
\end{equation}
\end{enumerate}

Then, summing up all the bounds~(\ref{term2a}),~(\ref{term2b}),~(\ref{term2c}),~(\ref{term2d}) and ~(\ref{term3}), using  the equivalent norms $\|\nabla \mbox{\boldmath$e$}_c\|_{0,\Omega}$ and $\| \mbox{\boldmath$\varepsilon$}{({\bm e}_c)}\|_{0,\Omega}$ in ${H}^1_0(\Omega)^d$, factorize the term $\nabla {\bm e}_c$ and finally deviding both sides of the resulting inequality by  $\nabla {\bm e}_c$, we arrive at the energy  bound for the conforming part of the error as follows:
  \begin{equation}\label{energyno}
  \begin{split}
\|\nabla \mbox{\boldmath$e$}_c\|_{0,\Omega}\leq C \bigg( \sum_{K \in {\mathcal{T}}_h}{h}_K^2 \|\mbox{\boldmath$r$}+\nabla\cdot \mbox{\boldmath$\sigma$}({{\bm z}_h})\|^2_{0,K}
&+ \sum_{e\in {\mathcal E}^I_{h}}{\sf h}\|\llbracket{({\bm \sigma}({{\bm z}_h})) {\bm \nu}_e}\rrbracket\|^2_{0,e}\\
&+\sum_{e\in {\mathcal{E}}_h}\alpha {\sf h}^{-1}\|\llbracket{{\bm z}_h}\rrbracket\|^2_{0,e} \bigg)^{1/2}.
\end{split}
\end{equation}

Now to derive the error in $L^2(\Omega)$-norm, we will use triangular inequality and Poincar\'e's inequality. By triangular inequality we have that 
\begin{equation}
\|\mbox{\boldmath$z$}- {{\bm z}_h}\|_{0,\Omega}\leq \|{\bm e}_c\|_{0,\Omega}+\|{\bm e}_d\|_{0,\Omega}.
\end{equation}
Here, by using Poincar\'e's inequality for the function ${\bm e}_c \in H^1_0(\Omega)^d$, we have that
$$\|{\bm e}_c\|_{0,\Omega}\leq C_{\rm F\Omega}\|\nabla{\bm e}_c\|_{0,\Omega}.$$
Then making use of the bound for the energy norm ~(\ref{energyno}), we derive the bound for $L^2$-norm of ${\bm e}_c$, and combining with Lemma~\ref{Ka} to bound the term $\|{\bm e}_d\|_{0,\Omega}$, we complete the proof.
 \end{proof}
  \begin{remark}
 {\em (Residual-based {\em a posteriori} error indicators)}.
 Notice that the sum of the above quantities can be decomposed into element-wise contributions. For example in case of the global estimator in Theorem \ref{IP} we can use the local error indicator: 
 \begin{equation*}
 \begin{split}
 &\eta_K=\\
 &\sqrt{{h}_K^2\|{\bm r}+\nabla \cdot ({\bm \sigma}({\bm z}_h))\|^2_{0,K}+ {\sf h}\|\llbracket{({\bm \sigma}({{\bm z}_h})) {\bm \nu}_e}\rrbracket\|^2_{0,\partial K \setminus \partial  \Omega}
+({\sf h}+\alpha {\sf h}^{-1})\|\llbracket{{\bm z}_h}\rrbracket\|^2_{0,\partial K} },
\end{split}
\end{equation*}
and note that these contributions provide local measures of the magnitude of the residual and are able to provide refinement indicators for adaptive strategies.
And in practice, because the mesh is taken with small $h_K$, then the term $\sf h$ is very small in comparison with $\alpha {\sf h}^{-1}$, so this term $({\sf h}+\alpha {\sf h}^{-1})\|\llbracket{{\bm z}_h}\rrbracket\|^2_{0,\partial K}$ can be considered as $\alpha {\sf h }^{-1}\|\llbracket{{\bm z}_h}\rrbracket\|^2_{0,\partial K}$ .
 \end{remark} 
 \begin{remark}
 We present here two alternative ways to derive {\em a posteriori} error bunds for stationary elasticity problem. In te first approach, we make use of the duality technique, which is well-known for deriving the $L^2-$norm of the errrors in finite elements methods. But notice that by this method, we need the regularity assumption of the adjoint equation. In the second approach, we obtain automatically a bound in $L^2-$ norm from the bound of energy norm, but we also recognize that this bound loses an order of the mesh size in comparison with the bound obtained by the duality method. 
 \end{remark}
\subsection{Semi-discrete {\em a posteriori} error estimates}\label{expliciterrorsemie}

 Now in view of Lemma~\ref{completing} and the error bound in stationary case in Theorem~\ref{IP2} or \ref{IP}, we obtain the following {\em a posteriori} error bound for semi-discrete problem.  
\begin{theorem} With the notations defined as in Definition~\ref{reconstw}, assume  that $\bm f$ is differentiable with respect to time. Then the following error bound holds
\begin{equation}
 \begin{split}
\|\bm e\|_{L^\infty (0,T;L^2 (\Omega)^d} &\leq  \| {\mathscr E}_{\rm IP}({\bm  u}_h,{ \bm g}, {\mathcal T}_h)\|_{L^\infty(0,T)}\\
&+2 \int_0^T  {\mathscr E}_{\rm IP}(\partial_t{\bm  u}_h,\partial_t{ \bm g}, {\mathcal T}_h){\rm d}t+\sqrt {2 }  {\mathscr E}_{\rm IP}({\bm  u}_h(0),{ \bm g}(0), {\mathcal T}_h)\\
&+\sqrt 2 \|{\bm u}_0 - {\bm  u}_h(0)\|_{0,\Omega} +2C_{\rm F\Omega}c_*^{-1/2}\|{\bm u}_1-\partial_t{\bm  u}_h(0)\|_{0,\Omega}.
 \end{split}
\end{equation}
with ${\bm g}={{\bm B} {\bm u}_h}-{\sf \Pi}_h{\bm f}+{\bm f}$, and ${\mathscr E}_{\rm IP}$ is given by~~(\ref{IP2}) or (\ref{IP}).
\end{theorem}
\section{Fully discrete error estimates}
\subsection{Fullly discrete formulation}
\subsubsection{Space-time meshes and spaces}
We now discretize the time derivative by using the backward Euler discretization with a subdivison of the time interval $(0,T)$ into subintervals $(t^{n-1}, t^n), n=1,\ldots, N,$ with $t^0=0,$ and $t^N = T$, and with the uniform time-step $\tau=t^n-t^{n-1}$. Associated with the time subdivision,  let ${\mathcal T}_h^n$, $n = 0, \ldots, N$
be a sequence of meshes  which may be different from ${\mathcal T}_h^{n-1}$  when $n\geq 1$. We also assume that the sequences of meshes  ${\mathcal T}_h^{n}$, $n=0,\ldots,N$ are regular and shape regular, uniformly  in $n$.

 For each $n=0,\ldots,N$, we denote by ${\bf V}_h^n$  a DG finite element space of fixed degree $r$ built on the partition ${\mathcal T}_h^n$ 
$${\bf V}_h^n:= \{\mbox{\boldmath$v$}\in L^2 (\Omega)^d : \mbox{\boldmath$v$}|_K  \in {\mathbb P}_r ({K})^d,~ K \in {\mathcal T}_h^n\},$$
 and denote by ${\mathcal E}_h^n$ the union of all edges (or faces) of the triangulation ${\mathcal T}_h^n$.

\subsection{Fully discretization formulation}\label{schemef}
 The fully discretization reads as follows:
 For each $n = 1, \ldots ,N$, find ${\bm u}_{h}^ n \in {\bf V}^n_h$
such that
\begin{equation}\label{fu41}
( {\partial  }^2 {\bm u}_{h}^ n , {{\bm v}})_\Omega + a_h^n({\bm u}_{h}^ n , {{\bm v}} ) = ( \mbox{\boldmath$f$}^n ,{\bm  v})_\Omega
~~~\mbox{for all}~{{\bm v}} \in {{\bf V}}_h^n,
\end{equation}
where the bilinear form $a_h^n$ is given by
\begin{equation}\label{ahnbw}
\begin{split}
a_h^n(\mbox{\boldmath$u$},\mbox{\boldmath$v$} )&=\sum_{K\in {\mathcal{T}}_h^n} \int_{K}\mbox{\boldmath$\sigma$}({\bm u}):\mbox{\boldmath$\varepsilon$}({\bm v}){\rm d}{\bm x}-\sum_{e\in \mathcal{E}^n_h} \int_{e}\llrrbrace{(\mbox{\boldmath$\sigma$}({\bm u})) \mbox{\boldmath$\nu$}_e}\cdot\llbracket{\mbox{\boldmath$v$}}\rrbracket{\rm d}A\\
 &-\sum_{e\in \mathcal{E}^n_h} \int_{e}\llrrbrace{(\mbox{\boldmath$\sigma$}({\bm v})){\bm \nu}_e}\cdot\llbracket{\mbox{\boldmath$u$}}\rrbracket{\rm d}A+\sum_{e\in \mathcal{E}^n_h} \int_{e}{\texttt a}\llbracket{\mbox{\boldmath$u$}}\rrbracket\cdot\llbracket{\mbox{\boldmath$v$}}\rrbracket{\rm d}A;
 \end{split}
 \end{equation}
and $\mbox{\boldmath$f$}^ n := \mbox{\boldmath$f$} (t^n , \cdot)$,  the backward  finite difference scheme for the discrete temporal derivatives
\begin{equation}\label{backwardeuler2}
{ \partial}^ 2 {\bm u}_{h }^ n :=\frac{{\partial  } {\bm u}_{h}^ n - {\partial  } {\bm u}_h^{ n-1}}{\tau},
\end{equation}
with
\begin{equation}\label{backeu1}
{ \partial} {\bm u}_{h}^n:=
\frac{{\bm u}_{h}^n-{\bm u}_h^{n-1}}{\tau} ~~\mbox{for}~n=1,2,\ldots, N,
\end{equation}
and  the initial values ($n=0$) are given by: $${\bm u}_{h}^0 := {\sf \Pi}_h^0 \mbox{\boldmath$u$}_0, \quad \partial {\bm u}_h^0={\sf \Pi}_h^0\mbox{\boldmath$u$}_1,$$
here ${\sf \Pi}_h^0$ is the orthogonal $L^2$-projection onto the finite element space ${\bf V}_h^0$  (although other projections onto ${\bf V}_h^0$ can also be used for setting the values of  ${\bm u}_{h}^0$ and $\partial {\bm u}_h^0$).

It's known that the backward Euler scheme is unconditional stable, from the coercivity of the bilinear form  $a_h^n(\cdot,\cdot)$, we have that the problem~(\ref{fu41})
 admits a unique solution $({\bm u}_h^n)_{1\leq n\leq N}$ at each time step.

\subsection{Space-time reconstruction}\label{sectionspacetime}
With the notations of finite element meshes, spaces introduced in Section~\ref{schemef} for fully discrete formulation, we have the following definitions:
\begin{definition}\label{fulrec}({\em Space-time reconstruction}).
 \begin{enumerate}[i)]
 \item Let ${\bm u}_h^n, n=0,\ldots, N$, be the fully discrete solution computed by the method~(\ref{fu41}), ${\sf \Pi}^n_h:L^2(\Omega)^d\rightarrow {\bf V}^n_h$ be the orthogonal $L^2$-projection, and ${\bm B}^n: {\bf V}^n_h \rightarrow {\bf V}^n_h$ to be the discrete operator defined by
 \begin{equation}\label{defBn}
 \mbox{for}~~ {\bm z} \in {\bf V}^n_h, ( {\bm B}^n {\bm z}, {\bm v})_\Omega=a_h^n({\bm z},{\bm v}),~\forall {\bm v}\in {\bf V}^n_h.
 \end{equation}
 We define the {SE} reconstruction ${\bm w}^n\in H^1_{0}(\Omega)^d$ of ${\bm u}_h^n$ to be the solution of the {SE} problem
\begin{equation}\label{constwnpa}
a({\bm w}^n,\mbox{\boldmath$v$})=( \mbox{\boldmath$g$}^n,\mbox{\boldmath$v$})_\Omega ,~\forall \mbox{\boldmath$v$}\in H^1_{0}(\Omega)^d,
\end{equation}
with $\mbox{\boldmath$g$}^n={\bm B}^n {\bm u}_h^n-{\sf \Pi}^n_h\mbox{\boldmath$f$}^n+\widetilde{\mbox{\boldmath$f$}}^n$;
where ${\sf \Pi}^n_h$ be the orthogonal $L^2$- projection onto the space ${\bf V}_h^n$,  ${\widetilde{\mbox{\boldmath$f$}}}^0=\mbox{\boldmath$f$}(0,\cdot)$, ${\widetilde{\mbox{\boldmath$f$}}}^n=\tau^{-1}\int_{t^{n-1}}^{t^n}\mbox{\boldmath$f$}(t,\cdot){\rm d}t$,  for $n = 1, \ldots, N$. 

\item Define the reconstruction ${\partial  } {\bm w}^0 \in {  H}^1_0 (\Omega)^d$ to be the solution to the {SE} problem
\begin{equation}\label{partialw}
a({ \partial} {\bm w}^0, \mbox{\boldmath$v$}) = ( {\partial  } \mbox{\boldmath$g$}^0, \mbox{\boldmath$v$})_\Omega,~~~~
\forall \mbox{\boldmath$v$ }\in { H}^1_0(\Omega)^d,
\end{equation}
with
$${ \partial} \mbox{\boldmath$g$}^0 := {\bm B}^0 (\partial{\bm u}_h^0) - {\sf \Pi}^0_h \mbox{\boldmath$f$}^0 +\mbox{\boldmath$f$}^0,$$
recall that ${\sf \Pi}^0_h$ denotes the orthogonal $L^2$ projection onto the space ${\bf V}_h^0$ and $\partial{\bm u}_h^0={\sf \Pi}_h^0 {\bm u}_1$ as introduced in Section~\ref{schemef}.

\item The continuous time-extension ${\bm u}_{N} : [0, t^N] \times \Omega \rightarrow {\mathbb R}^d$ of $\{{\bm u}_h^ n \}_{n=0}^N$
is defined by
\begin{equation}\label{uN}
{\bm u}_{N}(t) :=
\frac{t - t^{n-1}}{\tau} {\bm u}_h^n+\frac{t^n - t}{\tau} {\bm u}_h^{n-1}-\frac{( t - t ^{n-1}) (t^n - t)^2}{\tau}{{\partial  } ^2{\bm u}_h^n},
\end{equation}
for $t \in (t^{n-1}, t^n ], n = 1, \ldots, N$, with ${\partial  }^2 {\bm u}_h^n$ given by backward Euler scheme~(\ref{backwardeuler2}). We
note that ${\bm u}_{N}$ is a $C^1$ function in the time variable, with ${\bm u}_{N}(t^n ) = {\bm u}_h^n$ and $\partial_t{\bm u}_{N} (t^n ) = {\partial  } {\bm u}_h^ n$ for $n = 0, 1, \ldots, N$.
\item
The time-continuous extension  $ \mbox{\boldmath$w$}_N:[0, t^N] \times \Omega \rightarrow {\mathbb R}^d$ of $\{{\bm w}^n\}_{n=0}^N$ is given by
\begin{equation}\label{wN}
 {\bm w}_N(t) :=\frac{t - t^{n-1}}{\tau} { {\bm w}}^n+\frac{t^n - t}{\tau} { {\bm w}}^{n-1}-\frac{( t - t ^{n-1}) (t^n - t)^2}{\tau}{ {\partial  }} ^2{{\bm w}^n},
 \end{equation}
for $t \in (t^{n-1}, t^n ], n = 1, \ldots, N$; with ${\partial  }^2 {\bm w}^n$ given by the backward Euler scheme as in~(\ref{backwardeuler2}), noting that ${ {\partial  } } {\bm w}^0$ is defined as in~(\ref{partialw}). By construction, this is also a $C^ 1$ function in the time variable.
\end{enumerate}
 \end{definition}
\begin{definition}\label{decomposefully}({\em Splitting the error}). Decomposing the error as follows:\\
\begin{equation}
\begin{split}
\mbox{\boldmath$e$}_N &:= {\bm u}_{N}- \mbox{\boldmath$u$} \\
&:= \mbox{\boldmath$\rho$}_N-\mbox{\boldmath$\theta$}_N,
\end{split}
\end{equation}
where $\mbox{\boldmath$\theta$}_N:= {\bm w}_N-{\bm u}_{N} $  and $\mbox{\boldmath$\rho$}_N := {\bm w}_N- \mbox{\boldmath$u$} \in H^1_0(\Omega)^d$.
\end{definition}
\subsection{Fully discrete error relation}\label{fullyrelationerror}
\begin{theorem}\label{fullyrelation}
{\em (Fully discrete error relation)}. Under the notation as in Definition~\ref{decomposefully}, for $t  \in (t^{n-1}, t^n], n = 1,\ldots, N$ we have
\begin{equation}
\begin{split}
( \partial^2_{tt}\mbox{\boldmath$e$}_N, \mbox{\boldmath$v$})_\Omega + a(\mbox{\boldmath$\rho$}_N, \mbox{\boldmath$v$}) = &( ({\sf I} -{\sf \Pi}_h ^n )\partial^2_{tt}{\bm u}_{N}, \mbox{\boldmath$v$}) _\Omega+  \mu^n (t) ( {\partial  }^2{\bm  u}_h^ n,{\sf \Pi}_h^n \mbox{\boldmath$v$})_\Omega \\
&+ a({\bm w}_N-{\bm w}^n, \mbox{\boldmath$v$}) 
+( \widetilde{\mbox{\boldmath$f$}}^n-\mbox{\boldmath$f$},\mbox{\boldmath$v$})_\Omega, ~{\mbox {for all}}~~ \mbox{\boldmath$v$}\in { H}^1_0(\Omega)^d;
\end{split}
\end{equation}
 where $\sf I$ is the identity mapping in ${ L}^2 (\Omega)^d$, and
 $\mu^n(t):=-6\tau (t - \frac{t^{n}+t^{n-1}}{2}).$
\end{theorem}
\begin{proof}

 Noting that $\partial^2_{tt}{\bm u}_{N} (t) = (1 + \mu^n (t)){\partial}^ 2 {\bm u}_h^ n$ for $t \in (t^{n-1}, t^n ], n = 1, \ldots, N$; then for all ${\bm v} \in H^1_0(\Omega)^d$ we have: 
\begin{equation*}
 \begin{split}
  (\partial^2_{tt}\mbox{\boldmath$e$}_N, {\bm v})_\Omega + a({\bm \rho}_N, {\bm v})& = ( \partial^2_{tt}{\bm u}_{N}, {\bm v})_\Omega + a({\bm w}_N,{\bm v}) -( {\bm f},{\bm v})_\Omega \\
&=(  ({\sf I} - {\sf \Pi}^ n_h )\partial^2_{tt}{\bm u}_{N}, {\bm v})_\Omega + (\partial^2_{tt}{\bm u}_{N}, {\sf \Pi}_h^ n {\bm v})_\Omega\\
& + a({\bm w}_N,{\bm v}) -( {\bm f},{\bm v})_\Omega \\
&=(  ({\sf I} - {\sf \Pi}^ n_h )\partial^2_{tt}{\bm u}_{N}, {\bm v})_\Omega + \mu^n(t) ( {\partial  }^ 2 {\bm u}_h^ n, {\sf \Pi}^ n_h {\bm v})_\Omega \\
&\quad - a_h^n({\bm u}_h^ n, {\sf \Pi}^ n_h  {\bm v}) + a( {\bm w}_N, {\bm v}) + ( {\sf \Pi}^n_h  {\bm f}^ n - {\bm f},{\bm v})_\Omega \\
& =(  ({\sf I} - {\sf \Pi}^ n_h )\partial^2_{tt}{\bm u}_{N},{\bm v})_\Omega + \mu^n(t) ( {\partial  }^ 2 {\bm u}_h^ n, {\sf \Pi}^ n_h {\bm v})_\Omega  \\
&+ a({\bm w}_N-{\bm w}^n, {\bm v}) +( \widetilde{{\bm f}}^n-{\bm f},{\bm v})_\Omega;
 \end{split}
\end{equation*}
noting that the first equation follows from the formulation~(\ref{weak}) of the weak solution $\bm u$. In the second and  third equalities, the properties of orthogonal $L^2$-projection ${\sf \Pi}_h^n$ on space ${\bf V}_h^n$ are employed,  and we also use the formulation of fully discrete scheme in the third equality, and the last equality comes from  the identity
$a_h^n({\bm u}_h^n, {\sf \Pi}_h^n \mbox{\boldmath$v$}) -( {\sf \Pi}_h^n \mbox{\boldmath$f$}^n- \widetilde{\mbox{\boldmath$f$}}^n,\mbox{\boldmath$v$})_\Omega = ( {\bm B}^n {\bm u}_h^n,{\sf \Pi}_h^n \mbox{\boldmath$v$})_\Omega -( {\sf \Pi}_h^n \mbox{\boldmath$f$}^n- \widetilde{\mbox{\boldmath$f$}}^n,\mbox{\boldmath$v$})_\Omega= a({\bm w}^n, \mbox{\boldmath$v$}),~~\forall {\bm v}\in H^1_0(\Omega)^d$.
\end{proof}
\begin{remark}\label{vanish}
({\em Zero-mean value property}). The particular form of the remainder $\mu^n (t)$ satisfies the
zero-mean value property
$$\int^{t^n}_
{t^{n-1}}
\mu^n (t) dt = 0, \quad n=1,\ldots, N,
$$
which is an important point for estimating the {\em a posteriori} bound derived from  the term $\mu^n(t) ( {\partial  }^ 2 {\bm u}_h^ n, {\sf \Pi}^ n_h {\bm v})_\Omega$ in the error relation~(\ref{fullyrelation}) above (see later in Lemma~\ref{I3moment} for the proof of time-reconstruction error bound).
\end{remark}
\subsection{Abstract fully {\em a posteriori} error bound}\label{abstractfullyerrorapos}
To analyze the error, we introduce the following quantities  that will provide
error estimator of the fully discrete scheme in Theorem~\ref{ful}. These indicators are originally introduced by Lakkis et al. in ~\cite{G-L-M} concerning {\em a posteriori} error estimates for finite element method to the wave equation.
\begin{definition}\label{indicators}
 {\em (A posteriori error indicators)}.   We define
\begin{enumerate}[i)]
 \item  The mesh change indicator is given by  
 $$\zeta_{\rm {MC}} :=\sum_{n=1}^{N}\int_{t^{n-1}}^{t^n} \|({\sf I}- {{\sf \Pi}}_h^n)\partial_t{\bm u}_{N}\|_{0,\Omega} {\rm d}t+\sum_{n=1}^{N-1}({t^N}-t^n) \|({\sf \Pi}_h^{n+1}- {\sf \Pi}_h^n){{ \partial} \bm u}_h^n\|_{0,\Omega}.$$

\item The evolution error indicator reads
 $$\zeta_{\rm{evo}}  :=\int_{0}^{{t^N}}\|\mathcal{G}\|_{0,\Omega}{\rm d}t,$$
 where ${\mathcal G}: (0,t^N] \rightarrow {\mathbb R}^d$ with ${\mathcal G}|_{(t^{n-1},t^n]}:={\mathcal G}^n$, $n=1,\ldots,N$ and
 \begin{equation}\label{defG}
 {\mathcal G}^n(t):=\frac{(t^n-t)^2}{2}{\partial  } {\bm g}^n-\left (\frac{(t^n-t)^4}{4\tau}-\frac{(t^n-t)^3}{3}\right ){\partial  }^2\mbox{\boldmath$g$}^n-{\bm \gamma}_n,
 \end{equation}
 with
 $\mbox{\boldmath$g$}^n$ as in Definition~\ref{fulrec} and ${\bm \gamma}_n:={\bm \gamma}_{n-1}+(\tau^2/2){\partial  } \mbox{\boldmath$g$}^n+(\tau^3/12){\partial  }^2 \mbox{\boldmath$g$}^n$, $n=1,\ldots,N$
with ${\bm \gamma}_0=\bm 0$;
\item The data error indicators are given by
 $$\zeta_{\rm {osc}}  :=\frac{1}{2\pi}\sum_{n=1}^{N}\left (\int_{t^{n-1}}^{t^n} \tau^3 \|{\widetilde{\bm  f}}^n-{\bm f}\|^2_{0,\Omega}{\rm d}t \right)^{1/2},$$
 which can be viewed as an error estimator related to the time-oscillation of the source term.

 \item The time reconstruction error indicators
  $$\zeta_{{\rm{T.Rec}}} := \frac{1}{2\pi}\sum_{n=1}^{N-1}\left (\int_{t^{n-1}}^{t^n} \tau^3 \|\mu^n{\partial  }^2{\bm u}_h^n\|^2_{0,\Omega} {\rm d}t \right)^{1/2}.$$
\end{enumerate}
\end{definition}
\begin{theorem}\label{ful}
  {\rm (Abstract fully discrete error bound)}. Let $\bm u$ be the  weak solution to ~(\ref{weak}); ${\bm u}_N$ and ${\bm w}_N$ are reconstructed from the fully discrete solution $\{{\bm u}_h^n\}_{n=1}^N$ and its { SE} reconstruction $\{{\bm w}^n\}_{n=1}^N$, respectively, as in~(\ref{uN}) and~(\ref{wN}),  and recalling the notation of 
indicators in Definition~\ref{indicators} then the following {\em a posteriori} error estimate holds
\begin{equation}
\begin{split}
\|\mbox{\boldmath$e$}_N\|_{{ L}^\infty (0,t^N;{ L}^2 (\Omega)^d)}& \leq \|\mbox{\boldmath$\theta$}_N\|_{{ L}^\infty (0,t^N;{L}^2 (\Omega)^d)}+\sqrt 2\|\mbox{\boldmath$\theta$}_N(0)\| _{0,\Omega} +2 \int_0^{t^N} \|\partial_t \mbox{\boldmath$\theta$}_N\| _{0,\Omega}{\rm d}t\\
&+2\big ({\zeta}_{\rm{MC}}+ {\zeta}_{\rm{evo}}+ {\zeta}_{\rm{osc}}+ {\zeta}_{\rm{T.Rec}}\big)\\
&+ \sqrt 2 \|\mbox{\boldmath$u$}_{0} - {\bm u}_h^0\| _{0,\Omega}+2{C_{F\Omega} c_*^{-1/2}}\|\mbox{\boldmath$u$}_1-\partial{\bm u}_h^0\| _{0,\Omega},
\end{split}
\end{equation}
where $C_{\rm F\Omega}$ is the constant of the Poincar\'e inequality, ${\bm u}_h^0$ and $\partial{\bm u}_h^0$ are the orthogonal $L^2-$projections onto the space ${\bf V}_h^0$ defined in Section~\ref{schemef}.
\end{theorem}
\begin{proof}
The derivation  of this error bound is  similar to  in semi-discrete case. Firstly, we will take a test function ${\bm v}= \hat{\bm v}_N$ with $ \hat{\bm v}_N$ is defined similarly as in~(\ref{testfunction})
\begin{equation}
\hat{\mbox{\boldmath$v$}}_N(t, \cdot) =
\int_t^{{t^*}} \mbox{\boldmath$\rho$}_N(s, \cdot) ds,~
t \in [0, t^N],
\end{equation}
 assuming that $t^{m-1}\leq {t^*} \leq  t^m$ for some integer $m$ with $1\leq m\leq N$, and ${\bm \rho}_N$ is defined as in ~(\ref{decomposefully}). Then $\hat{\mbox{\boldmath$v$}}_N \in { H}^1_{0}(\Omega)^d$ and ${\mbox{\boldmath$\rho$}}_N \in  { H}^1_{0}(\Omega)^d$, and  we observe that
$$\hat{\mbox{\boldmath$v$}}_N({t^*} , \cdot) = 0, ~~\mbox{ and}~~ \partial_t \hat{\mbox{\boldmath$v$}}_N(t, \cdot) = -\mbox{\boldmath$\rho$}_N(t, \cdot)~~~\mbox{ a.e. in}~[0,t^N].$$
 We integrate the resulting equation with respect to $t$ between $0$ and ${t^*}$, to arrive at
\begin{equation}\label{1stequ}
\int_0^{t^*} ( \partial^2_{tt}{\bm e}_N, {\hat {\bm v}}_N)_\Omega {\rm d}t+\int_0^{t^*} a({\bm \rho}_N, {\hat {\bm v}}_N){\rm d}t=\sum_{i=1}^4{\mathcal I}_{i}({t^*}),
\end{equation}
where 
\begin{equation}\label{notationI1234}
\begin{split}
{\mathcal I}_{1}({t^*})&:=\sum_{n=1}^{m-1}\int_{t^{n-1}}^{t^n} (({\sf I}-{\sf \Pi}_h^n)\partial^2_{tt}{\bm u}_{N}, \hat{\bm v}_N)_{\Omega}{\rm d}t+\int_{t^{m-1}}^{{t^*}} (({\sf I}-{\sf \Pi}_h^m)\partial^2_{tt}{\bm u}_{N}, \hat{\bm v}_N)_{\Omega}{\rm d}t\\
{\mathcal I}_{2 }({t^*})&:==\sum_{n=1}^{m-1}\int_{t^{n-1}}^{t^n} a({\bm w}_N-{\bm w}^n,\hat{\bm v}_N){\rm d}t+\int_{t^{m-1}}^{{t^*}}a({\bm w}_N-{\bm w}^m,\hat{\bm v}_N){\rm d}t\\
{\mathcal I}_{3}({t^*})&:=\sum_{n=1}^{m-1}\int_{t^{n-1}}^{t^n} (\widetilde{\bm f}^n-{\bm f}, \hat{\bm v}_N)_{\Omega}{\rm d}t+\int_{t^{m-1}}^{{t^*}} (\widetilde{\bm f}^m-{\bm f}, \hat{\bm v}_N)_{\Omega}{\rm d}t,\\
{\mathcal I}_{4}({t^*})&:=\sum_{n=1}^{m-1}\int_{t^{n-1}}^{t^n} \mu^n (\partial ^2 {\bm u}_h^n, {\sf \Pi}_h^n \hat{\bm v}_N)_\Omega {\rm d}t+\int_{t^{m-1}}^{{t^*}} \mu^m (\partial ^2 {\bm u}_h^m, {\sf \Pi}_h^m \hat{\bm v}_N)_\Omega {\rm d}t
\end{split}
\end{equation}
Integrating by parts the first term of the left-hand side of Equation~(\ref{1stequ}), , and using the properties of $\hat{\bm v}_N$, we obtain 
\begin{equation}
\begin{split}
\int_{0}^{t^*} \frac{1}{2} \frac{d}{dt}\|{\bm \rho}_N\|^2_{0,\Omega}{\rm d}t-\int_{0}^{t^*} \frac{1}{2} \frac{d}{dt}a(\hat{\bm v}_N,\hat{\bm v}_N){\rm d}t&=\int_{0}^{t^*} (\partial_t {\bm \theta }_N, {\bm \rho}_N)_\Omega{\rm d}t\\
&+ ( \partial_t{\bm e}_N(0),\hat{\bm v}_N(0))_\Omega+\sum_{i=1}^4 {\mathcal I}_i({t^*}),
\end{split}
\end{equation}
which implies that
\begin{equation}
\begin{split}
\frac{1}{2}\|{\bm \rho}_N ({t^*})\|^2_{0,\Omega}-\frac{1}{2}\|{\bm \rho}_N(0)\|^2_{0,\Omega}+ \frac{1}{2} a(\hat{\bm v}_N(0),\hat{\bm v}_N(0))&=\int_{0}^{t^*} ( \partial_t{\bm \theta }_N, {\bm \rho}_N)_\Omega {\rm d}t\\
&+ ( \partial_t{\bm e}_N(0),\hat{\bm v}_N(0))_\Omega+\sum_{i=1}^4 {\mathcal I}_i({t^*}).
\end{split}
\end{equation}
On the other hand, with  the indicators $\zeta_{{\rm MC}}$, $\zeta_{{\rm evo}}$, $\zeta_{{\rm osc}}$ and $\zeta_{{\rm T.Rec}}$  introduced in Definition~\ref{indicators},  we have the following  bounds:
\begin{lemma}\label{I1234} The following inequalities hold:
\begin{equation}
\begin{split}
{\mathcal I}_1({t^*})&\leq \zeta_{{\rm MC}} \max_{0\leq t\leq T} \|{\bm \rho}_N(t)\|_{0,\Omega};\\
{\mathcal I}_2({t^*})&\leq \zeta_{{\rm evo}}\max_{0\leq t\leq T} \|{\bm \rho}_N(t)\|_{0,\Omega};\\
{\mathcal I}_3({t^*})&\leq \zeta_{{\rm osc}} \max_{0\leq t\leq T} \|{\bm \rho}_N(t)\|_{0,\Omega};\\
{\mathcal I}_4({t^*})&\leq \zeta_{{\rm T.Rec}}\max_{0\leq t\leq T} \|{\bm \rho}_N(t)\|_{0,\Omega};
\end{split}
\end{equation}
\end{lemma}
The proof for the above lemma follows from the applications of integration by parts, the commutation of orthogonal $L^2$-projection with time differentiation and time integration, and the zero averages of $\mu^n$ and $\widetilde{\bm f}^n-{\bm f}^n$ on the interval $[t^{n-1},t^n]$. These results for wave equation were presented in ~\cite{G-L-M}, but for the sake of completeness, we also present, even they are similar, the proofs for our elasticity problem  in Appendix~\ref{A1I1234}. 

Therefore we have
\begin{equation}
\begin{split}
\frac{1}{2}\|{\bm \rho}_N ({t^*})\|^2_{0,\Omega}&-\frac{1}{2}\|{\bm \rho}_N(0)\|^2_{0,\Omega}+ \frac{1}{2} a(\hat{\bm v}_N(0),\hat{\bm v}_N(0))\\
&\leq \max_{0\leq t\leq T} \|{\bm \rho}_N(t)\|_{0,\Omega} \bigg ( \int_{0}^{t^*} \|\partial_t {\bm \theta }_N\|_{0,\Omega}{\rm d}t+ {\zeta}_{\rm{MC}}+ {\zeta}_{\rm{evo}}\\
&\quad+ {\zeta}_{\rm{osc}}+ {\zeta}_{\rm{T.Rec}}\bigg )+ \|\partial_t{\bm e}_N(0)\|_{0,\Omega}\|\hat{\bm v}_N(0)\|_{0,\Omega}.
\end{split}
\end{equation}
We select  ${t^*}$ such that $\|{\bm \rho }_N( {t^*})\|_{0,\Omega}=\max_{0\leq t\leq t^N} \|{\bm \rho}_N(t)\|_{0,\Omega}$. Then following analogous arguments as in  the proof of the semidiscrete case, we end up to
\begin{equation}
\begin{split}
\|\mbox{\boldmath$e$}_N\|_{{ L}^\infty (0,t^N;{ L}^2 (\Omega)^d)}& \leq \|\mbox{\boldmath$\theta$}_N\|_{{ L}^\infty (0,T;{L}^2 (\Omega)^d)}+ \sqrt 2 (\|\mbox{\boldmath$u$}_{0} - {\bm u}_h^0\| _{0,\Omega}+\|\mbox{\boldmath$\theta$}_N(0)\| _{0,\Omega}) \\
&+2 \bigg (\int_0^{t^N} \|\partial_t\mbox{\boldmath$\theta$}_N\| _{0,\Omega}{\rm d}t+\zeta_{\rm MC}+\zeta_{\rm evo}\\
&+\zeta_{\rm osc}+\zeta_{\rm T.Rec}\bigg) +2{C_{\rm F\Omega} c_*^{-1/2}}\|\mbox{\boldmath$u$}_1-\partial{\bm u}_h^0\| _{0,\Omega}.
\end{split}
\end{equation}
The proof is then completed.
\end{proof}
\section{Fully discrete { a posteriori} error estimates}\label{explicitfullyboundapos}
To arrive at a practical {\em a posteriori} bound for the fully discrete scheme from the abstract error estimate in Theorem~\ref{ful} above,  it remains to bound the terms  $\|\mbox{\boldmath$\theta$}_N(0)\|_{0,\Omega}$, $\|\mbox{\boldmath$\theta$}_N\|_{{ L}^\infty (0,t^N;{ L}^2 (\Omega)^d)}$ and $\int_{0}^{t^N}\|\partial_t{\bm \theta}_N\|_{0,\Omega}{\rm d}t$. We will establish the bounds of these terms in Proposition~\ref{e0}, \ref{conf} and \ref{timeder} below. This  enables us to prove the following error estimate:
\begin{theorem}\label{maintheoremapo}
 ({\rm Fully discrete}  a posteriori {\rm bound}). With the same hypotheses and notations as in Theorem~\ref{ful}, we have 
 \begin{equation}
 \begin{split}
 \|{\bm e}_N\|_{L^\infty(0,t^N;L^2(\Omega)^d)} &\leq \zeta_{{\rm sp}}+\zeta_{\rm tp}+\zeta_{\rm IC},
 \end{split}
 \end{equation}
 where $\zeta_{{\rm sp}}$ mainly accounts  for the spatial error, $\zeta_{\rm tp}$ mainly accounts  for  the temporal error and $\zeta_{\rm IC}$ represents the initial conditions of the problem. They are given as follows:
  \begin{equation}
 \begin{split}
 \zeta_{{\rm sp}}&=\zeta_{{\rm sp},1}+\zeta_{{\rm sp},2}+\zeta_{{\rm sp},3},\\
 \zeta_{\rm tp}&=2\big(\zeta_{\rm MC}+\zeta_{\rm evo}+\zeta_{\rm osc}+\zeta_{\rm T.Rec}\big),\\
 \zeta_{{\rm {IC}}}&=\sqrt {2} \|{\bm u}_0-{\bm u}_h^0\|_{0,\Omega}+2{C_{\rm F\Omega} c_*^{-1/2}}\|\mbox{\boldmath$u$}_1-\partial{\bm u}_h^0\| _{0,\Omega},
 \end{split}
 \end{equation}
 where the temporal indicators $\zeta_{\rm MC}, \zeta_{\rm evo}, \zeta_{\rm osc}, \zeta_{\rm T.Rec}$ are defined in Definition~\ref{indicators}, and the spatial indicators $\zeta_{{\rm sp},1}$, $\zeta_{{\rm sp},2}$, $\zeta_{{\rm sp},3}$ are given by
 \begin{equation}
  \begin{split}
  \zeta_{{\rm sp},1}&={\sqrt 2}{\mathscr E}_{\rm IP}^0,\\
  \zeta_{{\rm sp},2}&=  \max \bigg\{\frac{4\tau}{27}{\mathscr E}_{\rm IP}(\partial{\bm u}_h^0, {\partial  }\mbox{\boldmath$g$}^0, {\mathcal{T}}_h^0 )\bigg \},
  3\max_{0\leq n\leq N}\big ({\mathscr E}_{\rm IP}^n+2C^2_{F \Omega }c_{*}^{-1}\|{\widetilde {\bm f}}^n-{\bm f}^n\|_{0,\Omega}\big),\\
    \zeta_{{\rm sp},3}&=\sum_{n=1}^N 2({\mathscr E}^n _{\rm IP}+{\mathscr E}^{n-1} _{\rm IP})+\sum_{n=1}^N 4 \tau C^2_{{\rm F\Omega}}c_{*}^{-1}\|{\partial  } {\bm f}^n-{\partial  } {\widetilde {\bm f}}^n\|_{0,\Omega},
 \end{split}
 \end{equation}
 where ${\mathscr E}_{\rm IP}^n := {\mathscr E}_{\rm IP}({\bm u}_h^n , {\bm B}^n {\bm u}_h^n -{\sf \Pi}^n_h{\bm f}^n + {\bm f}^n, {\mathcal T}_h^n)$, for all $0\leq n\leq N$,  ${\mathscr E}_{IP}$ being defined from~(\ref{IP2}) or~~(\ref{IP}).
\end{theorem}

\section{Proof of Theorem~\ref{maintheoremapo}}\label{proofoftheomain}
This section is devoted to the proofs of  Proposition~\ref{e0},~\ref{conf} and \ref{timeder}. Plugging these results into Theorem~\ref{ful}, we will get Theorem~\ref{maintheoremapo}.

At first, we clearly  have the following bound:
\begin{proposition}\label{e0} The following estimate holds
\begin{equation}
{\sqrt 2}\|\mbox{\boldmath$\theta$}_N(0)\|_{0,\Omega}  \leq \zeta_{{\rm sp},1},
\end{equation}
where $\zeta_{{\rm sp},1}$ is introduced in Theorem~\ref{maintheoremapo}.
\end{proposition}
Secondly, we establish an estimate  for the term $\|\mbox{\boldmath$\theta$}\|_{{ L}^\infty (0,t^N;{ L}^2 (\Omega)^d)}$ as follows:
\begin{proposition}\label{conf} The following estimate holds
$$\|\mbox{\boldmath$\theta$}_N\|_{{ L}^\infty (0,t^N;{ L}^2 (\Omega)^d)} \leq \zeta_{{\rm sp},2},$$
where $\zeta_{{\rm sp},2}$ is introduced in Theorem~\ref{maintheoremapo}.
\end{proposition}
\begin{proof}
According to the construction of ${\bm u}_{N}$ and ${\bm w}_N$ in Definition~\ref{fulrec}, for $t \in (t^{n-1}, t^n ], n = 1, \ldots , N$, we have the following expression
\begin{equation}\label{yie}
\begin{split}
\mbox{\boldmath$\theta$}_N(t)& ={\bm w}_N-{\bm u}_{N}\\
=&
\frac{t - t^{n-1}}{\tau} ({\bm w}^n-{\bm u}_h^n)+\frac{t^n-t}{\tau}({\bm w}^{n-1}-{\bm u}_h^{n-1})-\frac{(t - t^{n-1})(t^n-t)^2}{\tau} ({\partial  }^2{\bm w}^n-{\partial  }^2{\bm u}_h^n),
\end{split}
\end{equation}
which yields 
$$\|\mbox{\boldmath$\theta$}_N(t)\|_{0,\Omega}\leq 3\max_{0\leq n\leq N}\|{\bm w}^n-{\bm u}_h^n\|_{0,\Omega},$$
noting that
$$\max_{t\in (t^{n-1},t^n]}\frac{(t - t^{n-1})(t^n-t)^2}{\tau}= \frac{4\tau^2}{27}. $$

Now we need to estimate the terms $\|{\bm w}^n-{\bm u}_h^n\|_{0,\Omega}, 0\leq n\leq N$ and $\|{\partial  }{\bm w}^0-\partial{\bm u}_h^0\|_{0,\Omega}$. To estimate the term $\|{\bm w}^n - {\bm u}_h^ n\|_{0,\Omega}$, we will prove the following result
\begin{lemma}\label{wjuj} The following estimate holds for all $0\leq n\leq N$
\begin{equation}
\|{\bm w}^n - {\bm u}_h^ n\|_{0,\Omega} \leq {\mathscr E}_{\rm IP}^n+2C^2_{\rm F \Omega }c_{*}^{-1}\|{\widetilde {\bm f}}^n-{\bm f}^n\|_{0,\Omega},
\end{equation}
 with ${\mathscr E}_{\rm IP}^n := {\mathscr E}_{\rm IP}({\bm u}_h^n , {\bm B}^n {\bm u}_h^n -{\sf \Pi}^n_h{\bm f}^n + {\bm f}^n, {\mathcal T}_h^n)$, for all $0\leq n\leq N$ where ${\mathscr E}_{IP}$ is defined from~(\ref{IP2}) or~~(\ref{IP}).
\end{lemma}
\begin{proof}
To prove the result in Lemma~\ref{wjuj}, we  proceed as follows: first, we define ${\underline{\bm w}}^n \in H^1_0(\Omega)^d$ to be the solution to the SE problem
\begin{equation}\label{w*}
a({\underline{\bm w}}^n,{\bm  v}) = ( {\bm  B}^n{\bm  u}_h^ n - {\sf \Pi}^ n_h {\bm f}^ n +{\bm f}^ n ,{\bm v})_\Omega
\end{equation}
for $n = 0, 1, \ldots , N$. \\
Note that since $\widetilde{\bm f}^0={\bm f}^0$, we have  ${\underline{\bm w}}^0={\bm w}^0$. 
On the other hand, we know that  ${\bm u}_h^ n$ is the  DG approximation in  ${\bf V}^n_h$ of the {SE} problem~(\ref{w*}), noting that to prove this we follow the same proof as in Remark~\ref{rolew22}. Then in view of Theorem~\ref{IP2}, this implies
that
\begin{equation}\label{i1}
\|{\underline{\bm w}}^n-{\bm u}_h^n\|_{0,\Omega}\leq C {\mathscr E}_{\rm IP}({\bm u}_h^n , {\bm B}^n {\bm u}_h^ n -{\sf \Pi}^n_h{\bm f}^n + {\bm f}^n, {\mathcal T}_h^n),
\end{equation}
for $n =  0, \ldots , N$. \\

Second, we need to estimate   $\|{\bm w}^n-{\underline{\bm w}}^n\|_{0,\Omega}$. 
Observing that ${\bm w}^n -{\underline{\bm w}}^n$ is the solution of the stationary elasticity problem with the load ${\widetilde{\bm f}}^n-{\bm f}^n$:
$$a({\bm w}^n -{\underline{\bm w}}^n, {\bm v})=({\widetilde{\bm f}}^n-{\bm f}^n, {\bm v})_{\Omega},$$ we end up to
\begin{equation}\label{i3}
\|{\bm w}^n-{\underline{\bm w}}^n\|_{0,\Omega}\leq 2  C^2_{{\rm F\Omega}} c_*^{-1}\|{\widetilde{\bm f}}^n-{\bm f}^n\|_{0,\Omega},~~~\mbox{for}~~~n =  1, \ldots , N;
\end{equation}
due to the stability of {SE} problem.
Finally, thanks to the triangular inequality
\begin{equation}\label{paj2}
 \|{\bm w}^n-{\bm u}_h^n\|_{0,\Omega}\leq \|{\bm w}^n-{\underline{\bm w}}^n\|_{0,\Omega}+\|{\underline{\bm w}}^n- {\bm u}_h^n\|_{0,\Omega},
\end{equation}
along with the bounds  (\ref{i1}), (\ref{i3}) imply Lemma~\ref{wjuj}.
\end{proof}

Now it remains to estimate $\|{\partial   \bm w}^0-\partial{\bm u}_h^0 \|_{0,\Omega}$.
Similarly, we have
 $$a_h^0({\partial  }{\bm u}_h^0,{\bm v})=( {\bm B}^0 ( \partial{\bm u}_h^0), {\bm v})_{\Omega}=( {\bm B}^0 ( \partial{\bm u}_h^0)-{\sf \Pi}_h^0 {\bm f}^0+{\bm f}^0, {\bm v})_{\Omega},~\mbox{for all}~ {\bm v}\in {\bf V}^0_h.$$
Hence  comparing with the construction~(\ref{partialw}),  $\partial{\bm u}_h^0$ is the DG solution of ${\partial   {\bm w}}^0$, which yields
\begin{equation}\label{i2}
\|{\partial   \bm w}^0-{\partial  }{\bm u}_h^0\|_{0,\Omega}\leq {\mathscr{E}}_{\rm IP}(\partial{\bm u}_h^0, {\partial  {\bm g}}^0, {\mathcal T}_h^0).
\end{equation}
The proof of Proposition~\ref{conf} is thus completed.
\end{proof}

\begin{proposition}\label{timeder} The following estimate holds: 
 $$2\int_{0}^{t^N}\|\partial_t {\bm \theta}_N\|_{0,\Omega} {\rm d}t \leq \zeta_{{\rm sp},3},$$
 where $\zeta_{{\rm sp},3}$ is defined by (~\ref{maintheoremapo}).
\end{proposition}
\begin{proof}
Similarly to the proof of the previous proposition, the construction of ${\bm u}_{N}$ and ${\bm w}_N$ in Definition~\ref{fulrec} and ~(\ref{yie}), for $t\in (t^{n-1},t^n], n=1,\ldots, N$, implies that
$$\partial_t{\bm \theta}_N={ \partial}{\bm  w}^n-{ \partial}{\bm u}_h^n+\tau^{-1}(t^n-t)(3t-2t^{n-1}-t^n)({ \partial} ^2 {\bm w}^n-{\partial} ^2 {\bm u}_h^n),$$
from which we deduce that
\begin{equation}\label{ept}
 \int_{t^{n-1}}^{t^n}\|\partial_t{\bm \theta}_N\|_{0,\Omega}{\rm d}t \leq \tau  \|{\partial  }{\bm  w}^n-{\partial  }{\bm u}_h^n\|_{0,\Omega},
\end{equation}
noting that 
\begin{equation}\label{tjj0}
 \int_{t^{n-1}}^{t^n} (t^n-t)(3t-2t^{n-1}-t^n){\rm d}t=0.
\end{equation}
Summing up (\ref{ept}) for $n=1,\ldots,N,$ we get
\begin{equation}
 \int_0^{t^N}\|\partial_t{\bm \theta}_N\|_{0,\Omega} {\rm d}t\leq  \sum_{n=1}^N \tau \|{\partial  }{\bm  w}^n-{\partial  }{\bm u}_h^n\|_{0,\Omega}.
\end{equation}

{\em It remains to estimate the terms $\|{\partial  }{\bm  w}^n-{\partial  }{ \bm u}_h^n\|_{0,\Omega}$}. We will make use of the triangular inequality by combining the bounds for $\|{\partial  } \underline{\bm  w}^n-{\partial  }{ \bm u}_h^n\|_{0,\Omega}$ and $\|{\partial  }{\bm  w}^n-{\partial  }\underline { \bm w}^n\|_{0,\Omega}$. From the definition of $\{\underline { \bm w}^n\}_ {n=0}^N$ in ~(\ref{w*}), we obtain that  ${\partial  }\underline { \bm w}^n \in H^1_0(\Omega)^d,~~n=1,\ldots,N$ is the solution of the SE problem 
\begin{equation}\label{wunder}
a({\partial  }\underline { \bm w}^n,{\bm v})=(\partial ( {\bm B}^n { \bm u}_h^n)-\partial ({\sf \Pi}_h^n {\bm f}^n)+ \partial {\bm f}^n,{\bm v})_{\Omega},~\forall {\bm v}\in H^1_0(\Omega)^d.
\end{equation}
  We then have that ${\partial  }{\bm  w}^n-\partial \underline{\bm w}^n$ is the solution to the {SE} problem
 $$a({\partial  }{\bm  w}^n-\partial \underline{\bm w}^n,{\bm v})=({\partial  }{\widetilde{\bm f}}^n-{\partial  }{\bm f}^n,{\bm v})_{\Omega},~\forall {\bm v}\in H^1_0(\Omega)^d, ~~\mbox{for}~~ n=1,\ldots, N.$$
  Then from the stability of steady-state elasticity problem, we have
\begin{equation}\label{paj}
\|{\partial  }{\bm w}^n-{\partial  }{\underline{\bm w}}^n\|_{0,\Omega}\leq 2 C^2_{\rm F\Omega} c_*^{-1}\|{\partial  }{\widetilde{\bm f}}^n-{\partial  }{\bm f}^n\|_{0,\Omega},~~~\mbox{for}~n=1,\ldots,N.
\end{equation}

Now we derive a bound for $\|{\partial  } \underline{\bm  w}^n-{\partial  }{ \bm u}_h^n\|_{0,\Omega}$ in the following lemma.
\begin{lemma}\label{way1}
For $n=1,\ldots,N$, let ${\partial  }{\underline{\bm w}}^n \in H^1_0(\Omega)^d$ be the solution of problem~(\ref{wunder})  and ${\partial  }{\bm u}_h^n$ is defined by the backward Euler scheme~~(\ref{backeu1}) where ${\bm u}_h^n$ is the fully discrete solution from~(\ref{fu41}). The following error bound holds
\begin{equation}
\|{\partial  }{\underline{\bm w}}^n-{\partial  }{\bm u}_h^n\|_{0,\Omega}\leq \tau^{-1}({\mathscr E}_{\rm IP}^n+{\mathscr E}_{\rm IP}^{n-1}),
\end{equation}
with 
 ${\mathscr E}_{\rm IP}^n := {\mathscr E}_{\rm IP}({\bm u}_h^n , {\bm B}^n {\bm u}_h^n -{\sf \Pi}^n_h{\bm f}^n + {\bm f}^n, {\mathcal T}_h^n)$, for all $0\leq n\leq N$ and ${\mathscr E}_{IP}$ is defined by~(\ref{IP2}) or ~(\ref{IP}).
\end{lemma}
\begin{proof}
We again
denote by  $\underline{\bm w}^n$ the solution in~(\ref{w*}), for $i=0,1,\ldots, N$. From the backward finite difference scheme, we obtain the bounds  as follows:
\begin{equation}\label{partialw2}
\|{\partial  }{\underline{\bm w}}^n-{\partial  }{\bm u}_h^n\|_{0,\Omega}\leq \frac{1}{\tau}(\|{\underline{\bm w}}^n-{\bm u}_h^n\|_{0,\Omega}+\|{\underline{\bm w}}^{n-1}-{\bm u}_h^{n-1}\|_{0,\Omega}),
\end{equation}
for $n=1,\ldots,N.$
Employing~(\ref{i1}) allows us to end  the proof for Lemma~\ref{way1}.
\end{proof}
The proof of Proposition~\ref{timeder} is completed.
\end{proof}
\begin{remark}
In case of stationary mesh (i.e. the same mesh is used between the initial time $t^0$ and the final $t^N$, or ${\mathcal T}_h^n={\mathcal T}_h^{n-1}$, for all $n=1,\ldots, N$), we have an alternative result for Lemma~\ref{way1} as follows: 
\begin{equation}\label{rmstation}
\|{\partial  }{\underline{\bm w}}^n-{\partial  }{\bm u}_h^n\|_{0,\Omega}\leq  {\mathscr E}_{\rm IP}(\partial {\bm u}_h^n , \partial ({\bm B}^n {\bm u}_h^n) - \partial ({\sf \Pi}^n_h{\bm f}^n) + \partial {\bm f}^n, {\mathcal T}_h^n),
\end{equation}
 for all $1\leq n\leq N$ and ${\mathscr E}_{\rm IP}$ is defined from~(\ref{IP2}) or ~(\ref{IP}).
 \begin{proof}
  Indeed, from the definitions of $\{{\underline{\bm w}}^n\}_{n=0}^N$ in (\ref{w*}), we deduce  for all $n=1,\ldots, N$ that:
\begin{equation}\label{apartial}
a({\partial  }{\underline{\bm w}}^n, {\bm v})=(\partial ({\bm B}^n {\bm u}_h^n) - \partial ({\sf \Pi}^n_h{\bm f}^n) + \partial {\bm f}^n,{\bm v})_{\Omega}, \quad \forall {\bm v} \in {H^1_0(\Omega)}.
\end{equation}
On the other hand, from the definitions of the discrete operators $\{{\bm B}^n\}_{n=0}^N$ in (\ref{defBn}) and property of orthogonal $L^2$-projection,  we have that for all $n=1,\ldots, N$:
\begin{equation}
\begin{split}
a_h^n(\partial {\bm u}_h^n, {\bm v})&=(\partial ({\bm B}^n {\bm u}_h^n),{\bm v})_{\Omega},\quad \forall {\bm v} \in {\bf V}_h^n\\
&=(\partial ({\bm B}^n {\bm u}_h^n) - \partial ({\sf \Pi}^n_h{\bm f}^n) + \partial {\bm f}^n,{\bm v})_{\Omega}, \quad \forall {\bm v} \in {\bf V}_h^n,
\end{split}
\end{equation}
which implies that $\partial {\bm u}_h^n $ is the DG approximation in ${\bf V}_h^n$ of the boundary value problem~(\ref{apartial}), so we obtain the error estimate (\ref{rmstation}).
\end{proof}
\end{remark}
\section{Conclusion}
In this chapter, we have carried out an {\em a posteriori} error  analysis for the symmetric interior penalty Galerkin method for  the fully discretization of the time-dependent elasticity equation. The work is inspired by the method in ~\cite{G-L-M} which is used for the wave equation  in case of conforming  finite element method and is expanded to our elasticity problem in case of DG method. This method combines the {SE} reconstruction technique, the special testing procedure introduced by ~\cite{Baker},  and a suitable space-time reconstruction allows to derive  {\em a posteriori} error estimate for the time-dependent problem from the error of the auxiliary {SE} equation. 
We stress that this strategy can be adapted by many DG methods, as long as there exists {\em a posteriori} error estimate in $L^2-$norm of the corresponding DG method for stationary problem. The numerical implementation of the proposed bounds in the context of adaptive algorithm strategy will be considered in  upcoming works.

\section{Proof of Lemma~\ref{I1234}}\label{A1I1234}
In this section, we present the proofs of Lemma~\ref{I1234} in order to prove Theorem~\ref{ful}. This proof is similar to the proof carried out for the wave equation in~\cite{G-L-M}.
\begin{lemma}\label{prp1}{\em (Mesh change error estimate)}. Under the assumptions of Theorem~\ref{ful} and with the notation in~(\ref{notationI1234}), we have
$${\mathcal I}_1({t^*})\leq \zeta_{\rm MC} \max_{0\leq t\leq T} \|{\bm \rho}_N(t)\|_{0,\Omega}.$$
\end{lemma}
\begin{proof}
Observing that the projections ${\sf \Pi}_h^n$, $n=1,\ldots,N$ commute with the time differentiation, we integrate  by parts with respect to $t$, ending up to
\begin{equation}\label{I11}
\begin{split}
{\mathcal I}_1({t^*})&:=\sum_{n=1}^{m-1}\int_{t^{n-1}}^{t^n} (({\sf I}-{\sf \Pi}_h^n)\partial_t {\bm u}_{N}, {\bm \rho}_N)_{\Omega}{\mathrm d}t+\int_{t^{m-1}}^{{t^*}} \langle({\sf I}-{\sf \Pi}_h^m)\partial_t {\bm u}_{N}, {\bm \rho}_N)_{\Omega}{\mathrm d}t\\
&+\sum_{n=1}^{m-1} (({\sf \Pi}_h^{n+1}-{\sf \Pi}_h^n)\partial {\bm u}_{h}^n, \hat{\bm v}_N(t^n))_{\Omega}- (({\sf I}^0-{\sf \Pi}_h^0) \partial {\bm u}_h^0, \hat{\bm v}_N(0))_{\Omega}.
\end{split}
\end{equation}
The first two terms on the right hand side of~(\ref{I11}) are bounded by
\begin{equation}
\max_{0\leq t\leq T} \|{\bm \rho}_N(t)\|_{0,\Omega}\left(\sum_{n=1}^{m-1}\int_{t^{n-1}}^{t^n} \|({\sf I}- {{\sf \Pi}}_h^n)\partial_t {\bm u}_{N}\|_{0,\Omega}{\mathrm d}t+\int_{t^{m-1}}^{{t^*}}\|({\sf I}-{\sf \Pi}^m_h)\partial_t {\bm u}_{N}\|_{0,\Omega}{\mathrm d}t\right).
\end{equation}
Recalling the definition of ${\hat{\bm v}}_N$ and that of $\partial_t{\bm u}_{N}(t^n)=\partial {\bm u}_h^n$, $n=0,1,\ldots,N$, we can bound the last two terms in the right hand side of~(\ref{I11}) by
\begin{equation}\label{qu1}
\max_{0\leq t\leq T} \|{\bm \rho}_N(t)\|_{0,\Omega}\left(\sum_{n=1}^{m-1}({t^*}-t^n) \|({\sf \Pi}_h^{n+1}- {\sf \Pi}_h^n){{ \partial} \bm u}_h^n\|_{0,\Omega} +{t^*} \|({\sf I}-{\sf \Pi}^0_h)\partial {\bm u}_h^0\|_{0,\Omega}\right).
\end{equation}
Noting that $({\sf I}-{\sf \Pi}^0_h)\partial {\bm u}_h^0={\bm 0}$ because $\partial {\bm u}_h^0 \in {\bf V}^0_h$, and bounding~\ref{qu1}  with $t^*=T$, we obtain the result of Lemma~\ref{prp1}.
\end{proof}
\begin{lemma}\label{evolu}{\em (Evolution error bound)}.
Under the assumptions of Theorem~\ref{ful} and with the notation ~(\ref{notationI1234}), we have
$${\mathcal I}_2({t^*})\leq \zeta_{\rm evo} \max_{0\leq t\leq T} \|{\bm \rho}_N(t)\|_{0,\Omega}.$$
\end{lemma}
\begin{proof}
First we observe the identity
\begin{equation}
{\bm w}_N-{\bm w}^n=-(t^n-t)\partial {\bm w}^n+\Big ( \tau^{-1}(t^n-t)^3-(t^n-t^2)\Big) \partial ^2 {\bm w}^n
\end{equation}
for any $t$ in $(t^{n-1},t^n]$, $n=1,\ldots,m$. Hence from Definition~\ref{fulrec}, we deduce that
\begin{equation}\label{awwj}
a({\bm w}_N-{\bm w}^n,{\hat{\bm v}}_N)=( -(t^n-t)\partial {\bm g}^n+\Big ( \tau^{-1}(t^n-t)^3-(t^n-t^2)\Big) \partial ^2 {\bm g}^n,{\hat{\bm v}}_N)_\Omega.
\end{equation}
Then the integral of the first component in the inner product  on the right-hand side of~(\ref{awwj}) with respect to $t$ between $t^{n-1}$ and $t^n$ is given by ${\mathcal G}^n$ (see~(\ref{defG}).
Hence, integrating  by parts on each interval $(t^{n-1},t^n]$ gives rise to
\begin{equation}
{\mathcal I}_2({t^*})=\sum_{n=1}^m\int_{t^{n-1}}^{t^n}  ( {\mathcal G}^n,{\bm \rho}_N)_\Omega {\mathrm d}t +\int_{t^{m-1}}^{t^*} ( {\mathcal G}^{m},{\bm \rho}_N)_\Omega {\mathrm d}t,
\end{equation}
which  implies the expected result. Note that the choice of the constants ${\bm \gamma}_n$ in (\ref{defG}) makes  ${\mathcal G}$ continuous at $t^n$, $n=1,\ldots,N$, and we also have ${\mathcal G}(0)={\bm 0}$.
\end{proof}
\begin{lemma}\label{termosc}({\em Data approximation error estimate}). Under the assumptions of Theorem~\ref{ful} and with the notation~\ref{notationI1234}, we have
\begin{equation}
{\mathcal I}_3({t^*})\leq \zeta_{\rm osc} \max_{0\leq t\leq T} \|{\bm \rho}_N(t)\|_{0,\Omega}.
\end{equation}
\end{lemma}
\begin{proof}
We begin with observing the zero-mean value of ${\widetilde {\bm f}}^n-{\bm f}$ on $[t^{n-1},t^n]$ as follows
\begin{equation}
\int_{t^{n-1}}^{t^n}({\widetilde {\bm f}}^n-{\bm f}) {\mathrm d}t=0
\end{equation}
for all $n=1,\ldots,m-1$. Hence we have
\begin{equation}
\sum_{n=1}^{m-1}\int_{t^{n-1}}^{t^n}({\widetilde {\bm f}}^n-{\bm f},\hat {\bm v}_N)_\Omega {\mathrm d}t=\sum_{n=1}^{m-1}\int_{t^{n-1}}^{t^n}({\widetilde {\bm f}}^n-{\bm f},{\hat {\bm v}}_N-{\widetilde {\hat {\bm v}}}_N^n)_\Omega  {\mathrm d}t,
\end{equation}
where ${\widetilde{ \hat {\bm v}}}_N^n(\cdot):=\tau^{-1}\int_{t^{n-1}}^{t^n}{\hat {\bm v}}_N(t,\cdot){\mathrm d}t$. Using the Friedrich-Poincar\'e's inequality with respect to the variable $t$ we have
\begin{equation}
\int_{t^{n-1}}^{t^n}\|\hat {\bm v}_N-{\widetilde {\hat {\bm v}}}_N^n\|_{0,\Omega}^2 {\mathrm d}t\leq \frac{\tau^2}{4\pi^2}\int_{t^{n-1}}^{t^n} \|\partial_t \hat {\bm v}_N\|_{0,\Omega}^2 {\mathrm d}t
\end{equation}
and recalling that $\partial_t \hat {\bm v}_N=-{\bm \rho}_N$, we get
\begin{equation}
\begin{split}
\sum_{n=1}^{m-1}\int_{t^{n-1}}^{t^n}({\widetilde {\bm f}}^n-{\bm f},\hat {\bm v}_N)_\Omega {\mathrm d}t
&\leq \sum_{n=1}^{m-1}\bigg (\int_{t^{n-1}}^{t^n}\|{\widetilde {\bm f}}^n-{\bm f}\|^2_{0,\Omega }{\mathrm d}t\bigg)^{1/2}\bigg(\int_{t^{n-1}}^{t^n}\|\hat {\bm v}_N-{\widetilde {\hat {\bm v}}}_N\|^2_{0,\Omega}{\mathrm d}t\bigg)^{1/2}\\
&\leq  \frac{1}{2\pi}\sum_{n=1}^{m-1}\bigg (\int_{t^{n-1}}^{t^n}\|{\widetilde {\bm f}}^n-{\bm f}\|^2_{0,\Omega} {\mathrm d}t\bigg)^{1/2}\bigg(\int_{t^{n-1}}^{t^n}\tau^2\|\bm {\rho}_N\|^2_{0,\Omega} {\rm d}t\bigg)^{1/2}\\
&\leq  \frac{1}{2\pi}\sum_{n=1}^{m-1}\bigg (\tau^3\int_{t^{n-1}}^{t^n}\|{\widetilde {\bm f}}^n-{\bm f}\|^2_{0,\Omega} {\mathrm d}t\bigg)^{1/2}\max_{0\leq t\leq T}\|\bm {\rho}_N(t)\|_{0,\Omega}.
\end{split}
\end{equation}
For the remaining term in ${\mathcal I}_{3}$, we first observe that
\begin{equation}\label{36}
\int_{t^{m-1}}^{t^*} \|\hat {\bm v}_N\|^2_{0,\Omega}{\mathrm d}t \leq \int_{t^{m-1}}^{t^*} \tau \int_{t}^{t^*}\|{\bm \rho}_N\|^2_{0,\Omega} {\mathrm d}s {\mathrm d}t \leq \tau ^3 \max_{0\leq s\leq T}\|\bm {\rho}_N(s)\|_{0,\Omega}^2,
\end{equation}
which implies that
\begin{equation}
\int_{t^{m-1}}^{{t^*}}({\widetilde {\bm f}}^m-{\bm f},\hat {\bm v}_N)_\Omega  {\mathrm d}t\leq \bigg ( \int_{t^{m-1}}^{{t^*}} \tau^3\|{\widetilde {\bm f}}^m-{\bm f}\|_{0,\Omega}^2 {\rm d}t \bigg )^{1/2} \max_{0\leq t\leq T}\|\bm {\rho}_N(t)\|_{0,\Omega},
\end{equation}
and gives Lemma~\ref{termosc}.

\end{proof}
\begin{lemma}\label{I3moment}({\em Time-reconstruction error bound}).
 Under the assumptions of Theorem~\ref{ful} and with the notation~(\ref{notationI1234}), we have
\begin{equation}
{\mathcal I}_4({t^*})\leq \zeta_{\rm T.Rec}\max_{0\leq t\leq T} \|{\bm \rho}_N(t)\|_{0,\Omega}.
\end{equation}
\end{lemma}
\begin{proof}
 Noting that $\partial ^2 {\bm u}_h^n$ is piecewise constant, and recalling the zero-mean value~(\ref{vanish}) of  $\mu^n$ in $[t^{n-1},t^n],~n=1,\ldots,N$, we have
\begin{equation}
\sum_{n=1}^{m-1}\int_{t^{n-1}}^{t^n} \mu ^n( \partial ^2 {\bm u}_h^n, {\sf \Pi}_h^n {\hat {\bm v}_N})_{\Omega}{\rm d}t =\sum_{n=1}^{m-1}\int_{t^{n-1}}^{t^n} \mu ^n( \partial ^2 {\bm u}_h^n, {\sf \Pi}_h^n ({\hat {\bm v}_N}-{\widetilde {\hat {\bm v}}}_N^n))_{\Omega}{\rm d}t,
\end{equation}
where  ${\widetilde {\hat {\bm v}}}_N^n(\cdot):=\tau^{-1}\int_{t^{n-1}}^{t^n}\hat {\bm v}_N(t,\cdot){\mathrm d}t$.  Since the projections ${\sf \Pi}_h^n$, $n=1,\ldots,N$ commute with the time integration, we obtain
\begin{equation}
\begin{split}
\sum_{n=1}^{m-1}\int_{t^{n-1}}^{t^n} \mu ^n( \partial ^2 {\bm u}_h^n, {\sf \Pi}_h^n ({\hat {\bm v}_N}-&{\widetilde {\hat {\bm v}}}_N^n))_{\Omega}{\rm d}t\\
&\leq  \frac{1}{2\pi}\sum_{n=1}^{m-1}\bigg (\int_{t^{n-1}}^{t^n}\|\mu ^n  \partial ^2 {\bm u}_h^n\|^2_{0,\Omega}{\rm d}t\bigg)^{1/2}\bigg(\int_{t^{n-1}}^{t^n}\tau^2\ \|{\sf \Pi}_h^n \bm {\rho}_N\|^2_{0,\Omega} {\rm d}t\bigg)^{1/2}\\
& \leq\frac{1}{2\pi}\sum_{n=1}^{m-1}\bigg (\int_{t^{n-1}}^{t^n}\tau^3 \|\mu ^n  \partial ^2 {\bm u}_h^n\|^2_{0,\Omega}{\rm d}t\bigg)^{1/2}\max_{0\leq t\leq T} \|{\bm \rho}_N(t)\|_{0,\Omega}.
\end{split}
\end{equation}
For the last term in ${\mathcal I}_4$, by using an argument similar to~(\ref{36}), we have
\begin{equation}
\int_{t^{m-1}}^{{t^*}} ( \mu ^m \partial ^2 {\bm u}_h^m, {\sf \Pi}_h^m {\hat {\bm v}_N})_{\Omega}{\rm d}t
\leq \bigg (\int_{t^{m-1}}^{{t^*}}\tau^3 \|\mu ^m  \partial ^2 {\bm u}_h^m\|^2_{0,\Omega} {\rm d}t\bigg)^{1/2}\max_{0\leq t\leq T} \|{\bm \rho}_N(t)\|_{0,\Omega},
\end{equation}
and this completes the proof.
\end{proof}

\end{document}